\documentclass[12 pt, draftcls, onecolumn]{IEEEtran}
\usepackage{times} % assumes new font selection scheme installed
\usepackage{epsfig}
\usepackage{pgf,tikz}
\usetikzlibrary{arrows}
\usepackage{algpseudocode}
\usepackage{pgf,tikz}
\usepackage{mathrsfs}
\usepackage{bm}

\usepackage{enumitem}
\usepackage{caption}
\usepackage{booktabs}
\usepackage{longtable}
\usepackage{pgf,tikz}
\usepackage{mathrsfs}
\definecolor{qqqqcc}{rgb}{0.,0.,0.8}
\definecolor{ttttff}{rgb}{0.2,0.2,1.}
\definecolor{xdxdff}{rgb}{0.49,0.49,1.}

\usepackage{algorithm,algorithmicx,subfigure, latexsym, amsmath, amsfonts, amssymb, cite,amsthm}
\newtheorem{lemma}{Lemma}
\newtheorem{theorem}{Theorem}

\newtheorem{remark}{Remark}
\newtheorem{definition}{Definition}
\newtheorem{example}{Example}

\newtheorem{assumption}{Assumption}
\newtheorem{condition}{Condition}

\begin{document}
\title{Distributed Interval Optimization with Stochastic Zeroth-order Oracle}

\author{Yinghui~Wang, Xianlin~Zeng, Wenxiao~Zhao, and~Yiguang~Hong
%\thanks{Y. Wang, W. Zhao, and Y. Hong are with School of Mathematical Sciences, University of Chinese Academy of Sciences, and Key Laboratory of Systems and Control, Academy of Mathematics and Systems Science, Chinese Academy of Sciences, Beijing,  China. e-mail: (wangyinghuisdu@163.com, wxzhao@amss.ac.cn, yghong@iss.ac.cn).}% <-this % stops a space
%\thanks{X. Zeng is with Key Laboratory of Intelligent Control and Decision of Complex Systems, School of Automation, Beijing Institute of
%Technology, 100081, Beijing, China.e-mail: (xianlin.zeng@bit.edu.cn).}% <-this % stops a space
%\thanks{Manuscript received November 21, 2018.}
}

\maketitle

\begin{abstract}
In this paper, we investigate a distributed interval optimization problem which is modeled with optimizing a sum of convex interval-valued  objective functions subject to global convex constraints, corresponding to
agents over a time-varying network. We first  reformulate the distributed interval optimization problem as a distributed constrained optimization problem by scalarization. Then, we design a stochastic zeroth-order algorithm to solve the reformulated distributed problem, optimal solutions of which are also proved to be Pareto optimal solutions of the distributed interval optimization problem. Moreover, we construct the explicit convergence and the convergence rate in expectation of the given algorithm. Finally, a numerical example is given to illustrate the effectiveness of the proposed algorithm.
%Then, we show that the Pareto optimal solutions of the reformulated the Pareto optimal solutions of the distributed constrained optimization problem is equivalent to Pareto optimal solutions to the distributed interval optimization problem, and we design a distributed subgradient-free algorithm for the distributed constrained optimization problems through constructing random differences of reformulated optimal objective functions. Moreover, we prove that for the proposed algorithm, a Pareto optimal solution can be achived almost surely over  the time-varying network. Finally, we give a numerical example to illustrate the effectiveness of the proposed algorithm.
\end{abstract}

\begin{IEEEkeywords}
distributed interval optimization, Pareto optimal solution, zeroth-order algorithm, convergence rate, time-varying network.
\end{IEEEkeywords}
\IEEEpeerreviewmaketitle

\section{Introduction}
\IEEEPARstart{R}{ecently}, distributed optimization and control in a network environment, where  agents only have the local information and exchange information with their neighbours, have attracted much attention, which maybe more effective in many large-scale problems than centralized designs.  In fact, distributed first-order algorithms (which require subgradient information of local objective functions) and second-order algorithms (which require Hessian matrices information of local objective functions) for various (constrained) optimization problems have been widely studied for sensor networks, smart grids, and computation, etc\cite{nedic2009distributed, yi2015distributed, nedic2010constrained, ram2010distributed, zhang2015distributed, cherukuri2015distributed, yuan2016regularized, yuan2017adaptive}. Also, when the computation of first-order and second-order information of local objective functions is expensive, (distributed) zeroth-order or subgradient-free algorithms are designed (referring to \cite{chen1999kiefer, conn2009introduction, duchi2013optimal, nesterov2011random, yuan2015randomized} and references therein). Note that the connectivity is a key issue in the distributed design.  Although fixed topologies are still required for distributed optimization designs in some situations,  time-varying jointly connected networks have been considered in many algorithms such as \cite{nedic2009distributed, nedic2010constrained, wang2018distributed, wang2017distributed, li2017distributed}.

%distributed optimization and control in a network environment have attracted much attention, which are much effective in many large-scale problems than the centralized designs, in the case when agents only have the local information and exchange information over the network with their neighbours.

%{\color{blue}Howerver, in many situations, such as in machine learning\cite{bhurjee2012efficient}, the computation of first-order and second-order information of local objective functions is expensive to evaluate.  Motivated by such background, (distributed) stochastic subgradient-free/zeroth-order algorithms are designed (referring to \cite{chen1999kiefer, conn2009introduction, duchi2013optimal, nesterov2011random} and references therein), where  stochastic ideas are employed in to guarantee the almost sure convergence and stability of algorithms. %This provides a landscope for distributed subgradient-free/zeros-order learning.
%On one hand, randomization methods proved to be powerful tools in systems and control, which may sacrifice some properties to improve performance (for example, in the control of uncertain systems, they are useful ingredients compared to the classical robustness methods \cite{tempo2012randomized}). In fact, they are quite natural and can improve the overall performance of the system or provide a feasible design, particularly for large-scale systems or distributed design over networks \cite{ishii2010distributed}.

In practice, local objective functions and constraints may not be accurately or explicitly described. For example, various uncertainties appear in power systems and related for operational security \cite{wu2012comparison}. Interval optimization is an approach for dealing with these uncertainties.  To solve optimization problems with uncertainties, the interval optimization problem (IOP), first proposed by \cite{neumaier1990interval} and further studied in \cite{rohn1994positive, levin1999nonlinear}, has been widely studied in many different areas such as economics \cite{hu2006novela} and power systems \cite{wu2012comparison}.  In the interval optimization problem setup, objective functions are interval-valued, which are described by intervals rather than real numbers. The well-defined partial orderings and convexity of interval-valued maps \cite{hisao1990multiobjective, hu2006novela, wu2008interval}  provide existence guarantees of solutions of  maximization and minimization of interval optimization problems. Up to now, the literature (referring to \cite{liu2007numerical, jiang2008nonlinear, jayswal2011sufficiency, hladik2012interval}) has provided various programming methods, including Wolfe's method and Lamke's algorithm, to deal with centralized interval optimization problems.

%Motivated by such background, the interval optimization is studied in \cite{neumaier1990interval, rohn1994positive, levin1999nonlinear}, which provides a framework to capture the uncertainty in optimization.In fact, the interval optimization problem (IOP), first proposed by \cite{neumaier1990interval} and further studied in \cite{rohn1994positive, levin1999nonlinear} and references therein, has been widely studied in many different research areas including economic \cite{hu2006novela} and power systems \cite{wu2012comparison}.  In the interval optimization problem setup, the objective functions are interval-valued, which are described by intervals rather than real numbers with interval-valued maps. The well-defined partial orderings and convexity of interval-valued maps \cite{hisao1990multiobjective, hu2006novela, wu2008interval},  provide existance guarantees of solutions of  maximization and minimization of interval optimization problems. Up to now, the literature (referring to \cite{liu2007numerical, jiang2008nonlinear, jayswal2011sufficiency, hladik2012interval}) has provided various programming methods, particularly based on Wolfe's method or Lamke's algorithm, to deal with centralized interval optimization problems.

With this background, it is nature for us to consider how to effectively construct distributed algorithms for interval optimization problems over (time-varying) multi-agent networks.  However, the partial order resulting from  intervals makes the method based on gradients of objective functions become hard, especially when we only have local information in a distributed design, and in some cases, the subgradient of interval-valued objective functions may not be available.  In fact, very few works were even done for centralized interval optimization without subgradients in the algorithm design. Up to now, although there are some works on distributed optimization problems without subgradient information of local objective functions, there is no zeroth-order design on distributed interval optimization without using subgradients of objective functions.

The motivation of this paper is to propose a distributed stochastic zeroth-order algorithm for interval optimization problems.  Due to difficulties in distributed interval optimization (including that the gradient/subgradient information of interval-valued  functions is, sometimes, computationally costly and even impracticable for some cases \cite{bhurjee2012efficient}), we actively employ a stochastic idea to solve distributed interval-valued problems.  In fact, stochastic methods provide a way for subgradient-free designs to overcome the difficulty  of obtaining subgradient information of local interval-valued functions. Also, stochastic ideas are employed to guarantee the almost sure convergence and stability of algorithms.  Here we propose a distributed  stochastic zeroth-order algorithm  for a class of interval optimization problems. The contributions of this paper are summarized as follows:
\begin{itemize}
\item[(a)] Following the rapid development of data science and engineering systems, we extend the centralized interval optimization problem \cite{rohn1994positive, levin1999nonlinear} to a distributed one. In fact, we reformulate the distributed interval optimization problem as a distributed constrained non-smooth optimization problem.  In this reformulation, optimal solutions of the reformulated problem are equivalent to Pareto optimal solutions of the distributed interval optimization problem.   Distributed randomization methods can be conveniently implemented for the reformulation, while  the well-known versions such as Wolfe's and Lamke's methods cannot be easily extended to distributed versions due to the difficulty of step-size selections \cite{bellet2015distributed}.

 \item[(b)] We design a new distributed stochastic zeroth-order algorithm  for the reformulated  non-smooth problem, since the subgradient of the interval optimization problem is hard to be obtained. The algorithm adopts random differences to approximate subgradients of local reformulated objective functions, which is also different from many existing distributed stochastic zeroth-order or subgradient-free algorithms (c.f., \cite{anit2018distributed, hajinezhad2017zeroth, yuan2015randomized, yuan2015zeroth}) though it is consistent with those algorithms when the local objective function is smooth.

\item[(c)]With the proposed algorithm, we prove the achievement of the global minimization with probability one, and further provide its convergence rate  in expectation.  Moreover, the convergence results of the proposed algorithm match the best rate of  distributed  zeroth-order algorithms \cite{anit2018distributed, hajinezhad2017zeroth, yuan2015randomized, yuan2015zeroth} with diminishing step-sizes.

 \end{itemize}

The rest of the paper is organized as follows. Preliminaries related to the analysis of the distributed interval optimization problem are given in Section \ref{sec2}. Then the distributed interval optimization problem is formulated and the corresponding distributed algorithm is introduced in Section \ref{sec3}, while the proposed
algorithm is analyzed in Section \ref{sec4}. Following that, a numerical example is
given in Section \ref{sec5}. Finally, some concluding remarks are addressed in Section \ref{sec6}.

\section{Mathematical preliminaries}\label{sec2}

In this section, we introduce mathematical preliminaries about convex analysis \cite{hiriart2012fundamentals, clarke1998nonsmooth, nedic2010constrained}, probability theory \cite{durrett2010probability, polyak1987introduction} and interval optimization, respectively.

\subsection{Non-smooth analysis}\label{sec2.1}
%Here are some concepts on convex analysis \cite{hiriart2012fundamentals, clarke1998nonsmooth}.

Let $\mathcal{R}^{p}$ be the $p$-dimensional Euclidean space. Denote $ \mathcal{R}^{p}_{+}$ as its non-negative orthants. $\|\cdot\|$ denotes the Euclidean norm. Denote the set of all non-empty compact intervals of $\mathcal{R}$ by $\mathcal{C}(\mathcal{R})$.

\begin{definition}\label{Def1}
\cite{hiriart2012fundamentals}Let $f(x):\mathcal{R}^{p}\rightarrow \mathcal{R}$ be a non-smooth convex function.
Vector-valued function $\triangledown f (x)\in\partial f (x)\subset \mathcal{R}^{p}$ is called the subgradient of  $f(x)$ if for any  $x,y\in dom(f)$, the following inequality holds:
\begin{align*}
f(y)-f(x)-\big \langle\triangledown f(x), y-x\big\rangle\geqslant 0.
\end{align*}
\end{definition}

The next result is useful in the analysis of non-smooth functions.

\begin{lemma}\label{Lem1}\cite{clarke1998nonsmooth}(Lebourg's Mean Value Theorem) Let $x,y\in X$. Suppose $f(x):\mathcal{R}^{m}\rightarrow \mathcal{R}$ is Lipschitz on an open set containing line segment $[x,y]$. Then there exists a point $u\in (x,y )$ such that
\begin{align*}
f(x)-f(y)\in \langle \partial f(u), x-y \rangle.
\end{align*}
\end{lemma}

Then we summarize some  inequalities on Euclidean norm \cite{clarke1998nonsmooth, nedic2010constrained} to be used in this paper.
\begin{lemma}\cite{ram2010distributed}\label{Lem2}
Let $x_{1}, x_{2}, \ldots, x_{n}$ be vectors in $\mathcal{R}^{p}$. Then
\begin{align*}
\sum_{i=1}^{n}\Big\|x_{i}-\frac{1}{n}\sum_{i=1}^{n}x_{j}\Big\|^{2}\leqslant \sum_{i=1}^{n}\Big\|x_{i}-x\Big\|^{2}, \quad \forall x\in \mathcal{R}^{p}.
\end{align*}
\end{lemma}

Denote the projection of $x$ onto set $X$ by  $P_{X}(x)$, i.e., $P_{X}(x)=\arg\min_{y\in X} \big\|x-y \big\|$, where $X$ is a closed bounded convex set in $\mathcal{R}^{p}$. The following lemma introduces some results on projection operators:
\begin{lemma}\label{Lem3}
\cite{hiriart2012fundamentals, nedic2010constrained}Let $X$ be  a closed convex set in $\mathcal{R}^{p}$. Then for any $x\in \mathcal{R}^{p}$, it holds that
\begin{itemize}
\item[(a)]$\big\langle x-P_{X}(x), y-P_{X}(x)\big\rangle \leqslant 0$,  for all $y\in X$
\item[(b)]$\big\|P_{X}(x)-P_{X}(y) \| \leqslant \big\|x-y \big\|$, for all $x,y\in \mathcal{R}^{m}$.
\item[(c)]$\big\langle x-y, P_{X}(y)- P_{X}(x)\big\rangle \leqslant -\big\|P_{X}(x)-P_{X}(y) \big\|^{2}$, for all $y\in \mathcal{R}^{m}$.
\item[(d)] $\big\|x-P_{X}(x) \big\|^{2} +\big \|y-P_{X}(x)\big\|^{2} \leqslant \big\| x-y\big\|^{2}$, for any $y\in X$.
\end{itemize}
\end{lemma}

\subsection{Probability theory}\label{sec2.2}
Denote $(\Omega,\mathcal{F}, \mathbb{P})$ as the probability space, where $\Omega$ is the  whole event space, $\mathcal{F}$ is  the $\sigma$-algebra on $\Omega$, and $\mathbb{P}$ is  the probability measure on $(\Omega,\mathcal{F})$.

\begin{definition}\label{Def2}
 \cite{durrett2010probability}
 \begin{itemize}
 \item[(a)] $x_{1},x_{2},\ldots,x_{k}\ldots$ is a sequence of random variables (r. v.) in $(\Omega,\mathcal{F}, \mathbb{P})$. If $P(x_{k}\rightarrow x)=1$, then $x_{k}$ converges $x$ almost surely (a. s.).
 \item[(b)]$x_{1},x_{2},\ldots,x_{k}\ldots$ is a sequence of random variables (r. v.) in $(\Omega,\mathcal{F}, \mathbb{P})$. If  $\mathbb{E}\|x_{k}-x\|^{p}\rightarrow 0$, then $x_{k}$ converges to $x$ in $L^{p}$.
 \end{itemize}
\end{definition}

In $(\Omega,\mathcal{F}, \mathbb{P})$, denote $\{ F(k)\}_{k\geq1}$ as a sequence of increasing sub-$\sigma$-algebras on $\mathcal{F}$. $\{h(k)\}_{k\geq1}$, $\{v(k)\}_{k\geq1}$ and $\{w(k)\}_{k\geq1}$ are variable sequences in $\mathcal{R}$ such that for each $k$, $h(k)$, $v(k)$ and $w(k)$ are $F(k)$-measurable.  The following lemma is for the convergence of super-martingales:

\begin{lemma}\label{Lem4}\cite{polyak1987introduction} Suppose that $\{v(k)\}_{k\geq1}$ and $\{w(k)\}_{k\geq1}$ are nonnegative and  $\sum_{k=1}^{\infty} w(k) <\infty$, and $\{h(k)\}_{k\geq1}$ is bounded from below uniformly. If
\begin{align*}
\mathbb{E}[h(k+1)| F(k)] \leqslant (1+\eta(k)) h(k)-v(k)+ w(k),\;\; \forall k\geqslant 1
\end{align*}
holds almost surely, where $\eta(k) \geqslant 0$ are constants with $\sum_{k=1}^{\infty}\eta(k) <\infty$, then $\{h(k)\}_{k\geq1}$ converges almost surely with $\sum_{k=1}^{\infty} v(k) < \infty$.
\end{lemma}

\subsection{Interval Optimization}\label{sec2.3}

Denote $A=[a_{L}, a_{R}]$ and $B=[b_{L}, b_{R}]$ as two non-empty compact intervals in $\mathcal{P}(\mathcal{R})$.  Then we introduce quasi-orderings on $\mathcal{C}(\mathcal{R})$ and  some properties of interval-valued maps.

\begin{definition}\label{Def3}\cite{hisao1990multiobjective, hu2006novela}
For any $A, B\in \mathcal{P}(\mathcal{R})$, denote
\begin{itemize}
\item[(a)]$A\leqq_{L} B$ if $a_{L}\leqslant b_{L}$;
\item[(b)]$A\leqq_{U} B$ if $a_{R}\leqslant b_{R}$;
\item[(c)]$A\leqq B$ if $A\leqq_{L} B$ and $A\leqq_{U} B$.
\end{itemize}
\end{definition}

\begin{definition}\label{Def4}\cite{hisao1990multiobjective, hu2006novela}
For any $A, B\in\mathcal{P}(\mathcal{R})$, denote
\begin{itemize}
\item[(a)]$A<_{L} B$ if $a_{L}< b_{L}$;
\item[(b)]$A<_{U} B$ if $a_{R}< b_{R}$;
\item[(c)]$A< B$ if $A<_{L} B$ and $A<_{U} B$;
\item[(d)]$A\leq B$ if  $A<_{L} B$ and $A\leqq_{U} B$, or  $A\leqq _{L} B$ and $A<_{U} B$.
\end{itemize}
\end{definition}
Let $G: \mathcal{R}^{p}\rightrightarrows\mathcal{R}$ be an interval-valued map with respect to $x$.  Then we introduce Lipschitz continuity and convexity of the map $G$.

\begin{definition}\label{Def6}\cite{aubin2012differential}
Let $G: \mathcal{R}^{p}\rightrightarrows\mathcal{R}$ be an interval-valued map. $G$ is locally Lipschitz at $x$ if there exist
$K > 0$ and a neighborhood $W$ of $x$ such that
\begin{align*}
G(x_{1})\subseteq G(x_{2})+K\|x_{1}-x_{2}\|,\quad \forall x_{1},x_{2}\in W.
\end{align*}
\end{definition}
In fact, $G$ is locally Lipschitz at $x$ if there exist a neighbourhood  $W$ of $x$ and a constant $K\geqslant 0$, such that
\begin{align*}
G(x_{1})\subseteq B\big(G(x_{2}),K\|x_{1}-x_{2}\|\big).
\end{align*}
Denote
\begin{align*}
B(A,\varrho)=\{y| d(y,A)\leqslant \varrho\},
\end{align*}
as the ball of radius $\varrho$ around subset $A$, where $y$ is chosen from a metric space.

\begin{definition}\label{Def7}\cite{wu2008interval} Let $G: \mathcal{R}^{p}\rightrightarrows\mathcal{R}^{q}$ be an interval-valued map. $G$ is convex (lower-convex or upper convex) on $\Omega$ if, $\forall x_{1},x_{2}\in \Omega$, $\forall \alpha\in [0,1]$,
\begin{align*}
G\big(\alpha x_{1}+(1-\alpha)x_{2}\big)\leqq (\leqq_{L} or \leqq{U})\alpha G(x_{1})+(1-\alpha)G(x_{2}).
\end{align*}
\end{definition}

\begin{remark}\label{Rem1} Suppose that $G$ is compact-valued and convex with $G(\cdot)=[L(\cdot),R(\cdot)]$. Then, by Definitions \ref{Def3} and \ref{Def4}, $L(\cdot)$, $R(\cdot):\mathcal{R}^{p}\rightarrow \mathcal{R}$ are convex functions with respect to $x\in \mathcal{R}^{p}$. Namely, for any $x_{1},\;x_{2}\in  \mathcal{R}^{p}$ and $t\in [0,1]$,  following inequalities hold:
\begin{align*}
L\big(tx_{1}+(1-t)x_{2}\big)&\leqslant tL(x_{1})+(1-t)L(x_{2}),\\
R\big(tx_{1}+(1-t)x_{2}\big)&\leqslant tR(x_{1})+(1-t)R(x_{2}).
\end{align*}
\end{remark}

Then let us consider interval optimization problems.
Let $G: \mathcal{R}^{p}\rightrightarrows\mathcal{R}$ be an interval-valued map. Now the interval optimization problem is given as follows:
\begin{align*}
(IOP) \quad\quad \min_{x}\quad G(x) \quad s.\; t.\; \quad x\in \Omega
\end{align*}
where $G(x)=[L(x),R(x)]$ is a non-empty compact interval in $\mathcal{R}$.  For illustration, we introduce an example of an interval valued function (\cite{bhurjee2012efficient}).

\begin{example}
Consider a function $G: \mathcal{R}\rightrightarrows\mathcal{R}$. Without loss of generality, consider $c$ as an order set, which is  influenced by orders maintained on the presence of components of $G(x)$.  If $G(x_{1},x_{2})=\big\{c_{1}x_{1}^{2}+c_{2}x_{1}e^{c_{3}x_{2}}:c_i\in C_i,\,i=1,2,3\big\}$, where $C_{i}=[c^{i}_{L},c^{i}_{R}],\,i=1,2,3,$ are intervals. Suppose $c=[c_{1},c_{2},c_{3}]^{\top}$, $t=[t_1, t_2, t_3]^{\top}$,  and $C(t)=[c_{1}(t_{1}),c_{2}(t_{2}),c_{3}(t_{3})]^{\top}$, where  $c_{i}(t_{i})=(1-t_{i})c^{i}_{L}+t_{i}c^{i}_{R}$ and $t_{i}\in [0,1]$ for $i=1,2,3$. For the given interval vector $C^{3}_{v}=\prod_{i=1}^3 C_{i}$, $G(x_{1},x_{2})=[L(x),R(x)]$ is an interval, where $L(x)=\min_{t\in [0,1]^3}  G_{c(t)}(x_{1},x_{2})$, $R(x)=\max_{t\in [0,1]^3}  G_{c(t)}(x_{1},x_{2})$, and $G_{c(t)}(x_{1},x_{2})=c_{1}(t_{1})x_{1}^{2}+c_{2}(t_{2})x_{1}e^{c_{3}(t_{3})x_{2}}$.
\end{example}

Recalling definitions of $L(x)$ and $R(x)$ of the example, we see that we cannot  get explicit expressions of $L(x)$ and $R(x)$, and this  IOP can be solved through set-valued optimization rather than vector valued optimization.

Based on quasi-orderings
of compact intervals in $\mathcal{C}(\mathcal{R})$ given in Definitions \ref{Def3} and \ref{Def4}, we define a Pareto optimal solution to IOP.

\begin{definition}\label{Def8}\cite{maeda2012optimization}
\begin{itemize}
 \item[(a)]A point  $x^{*} \in \Omega$ is said to be a solution to IOP  if $ G(x^{*}) \leqq G(x)$ for all $x\in \Omega$.
 \item[(b)]A point  $x^{*} \in \Omega$ is said to be a Pareto optimal solution to IOP if $G(\bar{x})\leqq G(x^{*})$ for some $\bar{x}\in \Omega$ implies $G(x^{*})\leqq G(\bar{x})$.
\end{itemize}
\end{definition}

Clearly, there is no solution to the interval optimization problem given in Fig. \ref{fig1}. However, $[x_{1},x_{2}]$ are Pareto optimal solutions to this given problem.
\begin{itemize}
\item[(a)]For $y\leqslant x_{1}$, we have $R(y)\geqslant R(x_{1})$ and $L(y)\geqslant L(x_{1})$, which means that $G(y)\geqq G(x_{1})$.
\item[(b)]For $y\geqslant x_{2}$, we have $R(y)\geqslant R(x_{2})$ and $L(y)\geqslant L(x_{2})$, which means that $G(y)\geqq G(x_{2})$.
\item[(c)]For $x_{1}\leqslant y\leqslant x_{2}$, we have $R(y)\leqslant R(x_{1})$, $L(y)\geqslant L(x_{1})$, $R(y)\geqslant R(x_{2})$ and $L(y)\leqslant L(x_{2})$ according to Definition \ref{Def8}. Therefore, $[x_{1},x_{2}]$ are Pareto optimal solutions to this given problem.
\end{itemize}
\begin{figure}[!htb]
  \centering
  \includegraphics[width=\hsize]{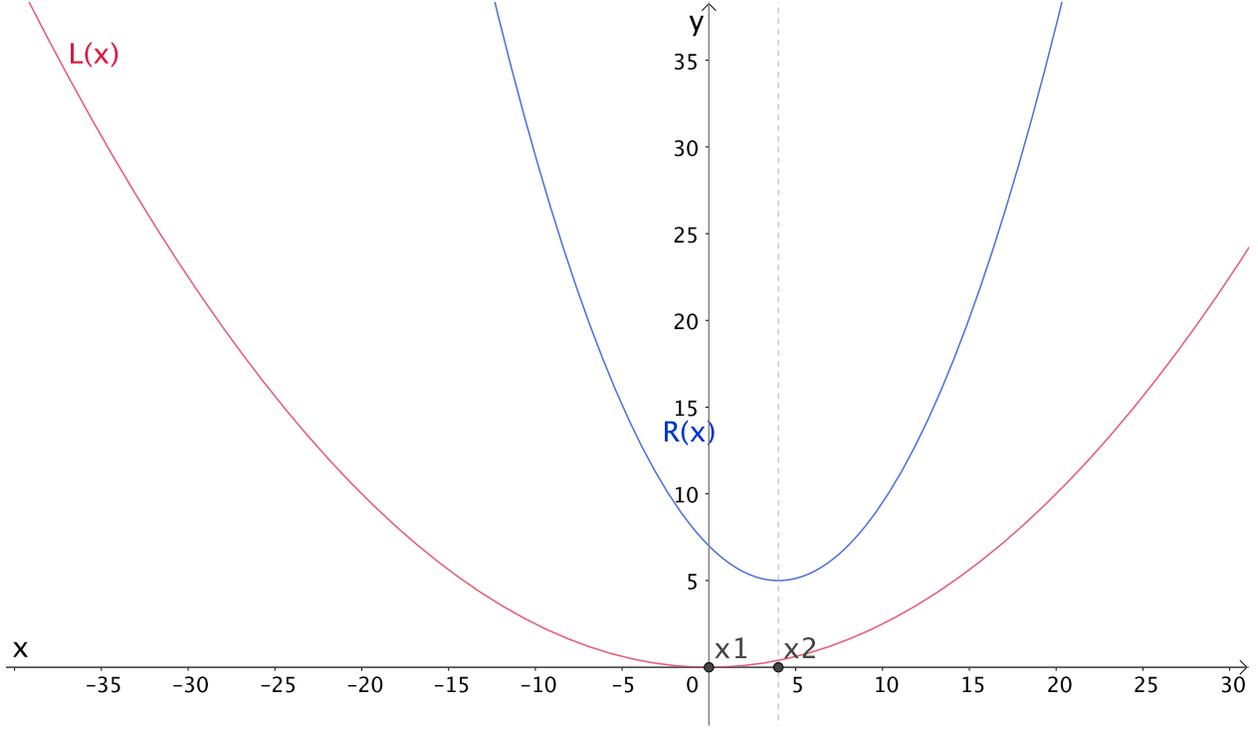}
  \caption{$L(x)$ and $R(x)$ for vector $x$.}
  \label{fig1}
\end{figure}

Associated with IOP, we  consider the following  scalarization of interval optimization problem:
\begin{align*}
SIOP: \quad \min_{x} \quad \;&\lambda L(x)+(1-\lambda)R(x)\notag\\
s.\; t.\; \quad &x\in \Omega
\end{align*}
where  $\lambda\in [0,1]$.

The following lemma is given in \cite{maeda2012optimization}. We give its proof here just for self-containment.

\begin{lemma}\label{Lem5} Suppose that $G$ is compact-valued and convex with respect to $x$:
\begin{itemize}
\item[(a)]If there exists a real number $\lambda\in(0,1)$ such that $x^{*}\in \Omega$ is an optimal solution to SIOP, then $x^{*}\in \Omega$ is a Pareto optimal solution to IOP.
\item[(b)]If $x^{*}\in \Omega$ is a Pareto optimal solution to IOP, then there exists a real number $\lambda\in [0,1]$ such that $x^{*}\in \Omega$ is an optimal solution to SIOP.
\end{itemize}
\end{lemma}

\begin{proof}
\begin{itemize}
\item[(a)]Given a real number $\lambda\in(0,1)$, let  $x^{*}\in \Omega$ be an optimal solution to SIOP. Suppose that there is $\bar{x}\in \Omega$ such that $G(\bar{x})\leqq G(x^{*})$, which implies $L(\bar{x})\leqslant L(x^{*})$ and  $R(\bar{x})\leqslant R(x^{*})$. Therefore,
\begin{align*}
\lambda L(\bar{x})+(1-\lambda) R(\bar{x})\leqslant \lambda L(x^{*})+(1-\lambda)R(x^{*}),
\end{align*}
which contradicts that $x^{*}$ is an optimal solution to SIOP.
\item[(b)]Let  $x^{*}\in \Omega$ be a Pareto optimal solution to IOP.
Since $G$ is compact-valued and convex with respect to $x$,
$L(x)$ and $U(x)$ are convex functions according to Remark \ref{Rem1}. Following Definition \ref{Def8}, there exists  a non-zero vector $\bm{\lambda}=[a,b]^{\top}$ with $a\geqslant 0$ and $b\geqslant 0$, such that
 \begin{align*}
 \bm{\lambda}^{\top}\begin{bmatrix}
L(x^{*})\\
R(x^{*})
\end{bmatrix}\leqslant \bm{\lambda}^{\top}\begin{bmatrix}
L(x)\\
R(x)
\end{bmatrix}
 \end{align*}
holds for all $x\in \Omega$. Define $\bar{\bm{\lambda}}=[\frac{a}{a+b},\frac{b}{a+b}]$ then
\begin{align*}
\bar{\bm{\lambda}}^{\top}\begin{bmatrix}
L(x^{*})\\
R(x^{*})
\end{bmatrix}\leqslant \bar{\bm{\lambda}}^{\top}\begin{bmatrix}
L(x)\\
R(x)
\end{bmatrix},
\end{align*}
which implies the conclusion.
\end{itemize}
\end{proof}

\section{Formulation and Algorithm}\label{sec3}

Consider the following distributed  interval optimization problem over an $n$-agent network:
\begin{align}\label{3-1}
(DIOP)\quad \min_{x} \quad G(\bm{x})&=\sum_{i=1}^{n}G_{i}(x_{i})\notag\\
s.\;t.\;\quad \quad x_{i}&=x_{j},\quad x_{i}\in X
\end{align}
where $\bm{x}=\big [x_{1}^{\top},x_{2}^{\top}, \ldots, x_{n}^{\top}\big ]^{\top}\in \mathcal{R}^{np}$, $x_{i}\in \mathcal{R}^{p}$, and $G_{i}:\mathcal{R}^{p}\rightrightarrows \mathcal{R}$ is a compact and convex interval-valued
function. In this setting, the state of an agent $i$ is the estimate of solution to  DIOP.  Each agent
$i$ knows the local function $G_{i}$  and global constraint $X$.

We make the following assumption on local functions and constraints for DIOP:
\begin{assumption}\label{Ass1}
\begin{itemize}
\item[(a)]$G_{i}(x)$ is a convex, compact, Lipschitz continuous interval-valued function.
\item[(b)]$X$ is a non-empty, compact, convex constraint set in $\mathcal{R}^{p}$.
\item[(c)]The subgradient of $G_{i}(x)$ is almost everywhere locally Lipschitz continuous.
\end{itemize}
\end{assumption}
Assumption \ref{Ass1}(a) is consistent with assumptions in the centralized case \cite{bhurjee2012efficient}, while Assumption \ref{Ass1}(b) is a quite common assumption for the boundedness of distributed and centralized stochastic algorithms based on diminishing step-sizes \cite{chen2006stochastic, ram2010distributed}.

Consider DIOP over a time-varying multi-agent  network, described by a directed graph $\mathcal{G}(k)=\big(\mathcal{N},
\mathcal{E}(k), W(k)\big)$, where $\mathcal{N} = \{1,2,...n\}$ is the agent set,  the edge set $\mathcal{E}(k)\subset
\mathcal{N}\times \mathcal{N} $ represents information communication at time $k$ and $W(k)=\big[w_{ij}(k)\big]_{ij}$ represents the adjacency matrix at time $k$.
Each agent interacts with its neighbors in $\mathcal{G}(k)=(\mathcal{N}, \mathcal{E}(k), W(k))$ at time $k$.  The following assumption is about communication topology

\begin{assumption}\label{Ass2}
The graph $\mathcal{G}(k)=\big(\mathcal{N},
\mathcal{E}(k), W(k)\big)$ satisfies:
\begin{itemize}
\item[(a)]There exists a constant $\eta$ with $0<\eta<1$ such that, $\forall k\geqslant 0$ and $\forall i, j $, $ w_{ii}(k)\geqslant \eta$; $w_{ij}(k)\geqslant \eta$ if $(j,i)\in \mathcal{E}(k)$.
\item[(b)]$W(k)$ is doubly stochastic, i. e. $\sum_{i=1}^{m}w_{ij}(k)=1$ and $\sum_{j=1}^{m}w_{ij}(k)=1$.
\item[(c)]There is an integer $\kappa\geqslant 1$ such that $\forall k\geqslant 0$ and $\forall (j,i)\in \mathcal{N}\times \mathcal{N}$,
\begin{equation*}
 (j,i)\in \mathcal{E}(k)\cup \mathcal{E}(k+1)\cup \cdots \cup \mathcal{E}(k+\kappa-1).
\end{equation*}
\end{itemize}
\end{assumption}

Assumption \ref{Ass2} reveals that agent $i$ can collect information from all its neighbors ``periodically". It is also a widely used
connectivity condition for distributed time-varying network designs (see \cite{nedic2010constrained,nedic2009distributed}).

Define the function $f:\mathcal{R}^{np}\times \mathcal{R}^{n} \rightarrow\mathcal{R}$ and $f_{i}:\mathcal{R}^{p}\times [0,1]\rightarrow \mathcal{R}$ as
\begin{align}
f\big (\bm{x},\bm{\lambda}\big )&\triangleq\sum_{i=1}^{n}f_{i}\big (x_{i},\lambda_{i}\big )\label{3-2}\\
f_{i}\big (x_{i},\lambda_{i}\big )&\triangleq \lambda_{i} L_{i}(x)+(1-\lambda_{i})R_{i}(x)\label{3-3}
\end{align}
where $i=1,2,\ldots,n$, $\bm{x}=\big [x_{1}^{\top},x_{2}^{\top}, \ldots, x_{n}^{\top}\big ]^{\top}\in \mathcal{R}^{nq}$ and $\bm{\lambda
}=\big [\lambda_{1},\lambda_{2}, \ldots, \lambda_{n}\big ]^{\top}\in \mathcal{R}^{n}$.
Note that both $ L(\bm{x})$ and $R(\bm{x})$ are separable, that is,
\begin{equation}\label{Rem2}
L(\bm{x})=\sum_{i=1}^{n}L_{i}(x_{i}),\quad R(\bm{x})=\sum_{i=1}^{n}R_{i}(x_{i}).
\end{equation}

 Let $\bm{\lambda}=\lambda_{0}\bm{1}_{n}$ with $\lambda_{0}\in (0,1)$.
We can write the distributed interval optimization \eqref{3-1} problem as:
 \begin{align}\label{3-4}
 \min_{x} \quad f\big (\bm{x},\bm{\lambda}\big )&=\sum_{i=1}^{n}f_{i}\big (x_{i},\lambda_{i}\big )\notag\\
s.\;t.\;\quad\quad\quad x_{i}&=x_{j},\quad x_{i}\in X\notag\\
\lambda_{i}&=\lambda_{j}
 \end{align}
where agent $i$ knows the information of $f_{i}$, $x_{i}$, $\lambda_{i}\in (0,1)$ and its neighborhood information.
Obviously, problem \eqref{3-4} degenerates to a conventional distributed constrained optimization problem \cite{ram2010distributed} when each agent $i$ choose a common parameter $\lambda_{i}=0$ or $\lambda_{i}=1$.  Some conclusions about the local objective function $f_{i}$ of \eqref{3-4} are listed in the following lemma.

\begin{lemma}\label{Lem13}\cite{bhurjee2012efficient,maeda2012optimization}
Suppose Assumption \ref{Ass1} holds.  Then, for $i=1,...,n $,
 \begin{itemize}
\item[(a)] $f_{i}\big (x,\lambda\big)$ is convex with respect to $x$, that is, for any $x_{1},x_{2}$,
 \begin{align*}
 f_{i}\big (\alpha x_{1}+(1-\alpha)x_{2},\lambda\big)\leqslant  \alpha f_{i}\big (x_{1},\lambda\big)+(1-\alpha) f_{i}\big (x_{2},\lambda\big),
  \end{align*}
  where $\alpha\in [0,1]$.
 \item[(b)] $f_{i}\big (x,\lambda\big)$ is convex with respect to $\lambda$.
  \item[(c)] $f_{i}\big (x,\lambda\big)$ is Lipschitz continuous with respect to $x$, that is, for any $x_{1},\;x_{2}$ and $\lambda$,
\begin{align*}
\big\|f_{i}\big (x_{1},\lambda\big)-f_{i}\big (x_{2},\lambda\big)\big\|\leqslant L\|x_{1}-x_{2}\|.
\end{align*}
  \item[(d)] $f_{i}\big (x,\lambda\big)$ is Lipschitz continuous with respect to $\lambda$, that is, for any $\lambda_{1},\;\lambda_{2}$ and $x$,
\begin{align*}
\big\| f_{i}\big (x,\lambda_{1}\big)-f_{i}\big (x,\lambda_{2}\big)\big\|\leqslant K\|\lambda_{1}-\lambda_{2}\|.
\end{align*}
\end{itemize}
\end{lemma}

The following lemma still holds for DIOP, whose proof is analogous to the proof of Lemma \ref{Lem5} and omitted here.
\begin{lemma}\label{Lem6}
If ($\bm{x}^{*},\bm{\lambda}^{*}) \in \mathcal{R}^{np}\times \mathcal{R}^{n}$ is an optimal solution to problem \eqref{3-4}, then $\bm{x}^{*}$ is a Pareto solution to problem \eqref{3-1}.
\end{lemma}

Since the differentiability of $f(\bm{x},\bm{\lambda})$ with respect to $x$ may not hold, we propose a distributed zeroth-order interval-valued algorithm \ref{Alg1} for problem \eqref{3-4}.
\begin{algorithm}[h]
\flushleft
\caption{ \bf Distributed stochastic zeroth-order  algorithm}
\label{Alg1}
\hspace*{0.02in} {\bf Input:}
Total numbers of iteration $T$, step-size $\iota(k)$. \\
{\bf Initialize:} $\xi_{i}\in X$ for all $i=1,2,\ldots n$.
\begin{algorithmic}[1]
\For{$k=0,\ldots T$}
\State  Average of local observations $x_{i}(k)$:
\begin{align}
 \xi_{i}(k)&=\sum_{j=1}^{n}w_{ij}(k)x_{j}(k).\label{3-5}
\end{align}
\State Calculation of local measurement $d_{i}(k)$
\begin{align}\label{3-6}
d_{i}(k)=\dfrac{\big[y_{i}^{+}(k)-y_{i}^{-}(k)\big]\bigtriangleup_{i}^{-}(k)}{2c(k)},
\end{align}
\State Descent Step:\begin{align}\label{3-7}
        \hat{\xi}_{i}(k)= \xi_{i}(k)-\iota(k)d_{i}(k).
    \end{align}
Projection Step:\begin{align}\label{3-8}
 x_{i}(k+1)=P_{X}\Big( \hat{\xi}_{i}(k)\Big).
\end{align}
\State Average of local observations $\lambda_{i}(k)$:
\begin{align}
 \lambda_{i}(k+1)&=\sum_{j=1}^{n}w_{ij}(k) \lambda_{i}(k)\label{3-9}.
\end{align}
  \EndFor
\end{algorithmic}
where $d_{i}(k)$ is used as an estimate for $\partial f_{i_{\xi_{i}(k)}}\big(\xi_{i}(k),\lambda_{i}(k)\big)$.
\end{algorithm}

In \eqref{3-6},
$\bigtriangleup_{i}(k)=\big[\bigtriangleup^{1}_{i}(k),\bigtriangleup^{2}_{i}(k),\ldots,\bigtriangleup^{p}_{i}(k)\big]^{\top}$.  $\bigtriangleup_{i}^{-}(k)=\Big[\frac{1}{\bigtriangleup_{i}^{1}(k)},\frac{1}{\bigtriangleup_{i}^{2}(k)},\ldots,\frac{1}{\bigtriangleup_{i}^{p}(k)}\Big]^{\top}$, where $\big\{\bigtriangleup_{i}^{q}(k)\big\}_{k\geqslant 0}$, $q=1,2,\ldots,p$, $k=1,2,\ldots$ is a sequence of mutually independent and identically distributed random variables with zero mean. The measurements $y_{i}^{+}(k)$ and $y_{i}^{-}(k)$ are given by
\begin{align*}
y_{i}^{+}(k)&=f_{i}\big(\xi_{i}(k)+c(k)\bigtriangleup_{i}(k),\lambda_{i}(k)\big),\\
y_{i}^{-}(k)&=f_{i}\big(\xi_{i}(k)-c(k)\bigtriangleup_{i}(k),\lambda_{i}(k)\big).
\end{align*}

Define $ F(k)=\sigma\big\{x_{i}(k),x_{i}(k-1),\ldots,x_{i}(0), i=1,2,\ldots,n; \lambda_{i}(k),\lambda_{i}(k-1),\ldots,\lambda_{i}(0),i=1,2,\ldots,n; \bigtriangleup_{i}(k-1), \bigtriangleup_{i}(k-2),\ldots, \bigtriangleup_{i}(0),\quad i=1,2,\ldots,n\big\}$, where $F(k)$ is  the $\sigma$-algebra created by the whole history of Algorithm \ref{Alg1} up to moment $k$ (referring to \cite{ram2010distributed}).

In Algorithm \ref{Alg1}, the following condition holds in the paper:
\begin{condition}\label{Con1}
\begin{itemize}
\item[(a)]  let $\big\{\bigtriangleup_{i}^{q}(k)\big\}_{k\geqslant 0}$ be a sequence of independent and identically distributed (i. i. d.) random variables for any fixed $(i, q)$, and for all $k\geqslant 0$ and $(i, q)$,
\begin{align*}
\big|\bigtriangleup_{i}^{q}(k)\big|<M_{1},\; \;\bigg|\frac{1}{\bigtriangleup_{i}^{q}(k)}\bigg|<M_{2},\;\;\mathbb{E}\bigg[\frac{1}{\bigtriangleup_{i}^{q}(k)}\bigg]=0;
\end{align*}
\item[(b)]$\big\{\bigtriangleup_{i}^{q}(k)\big\}_{k\geqslant 0}$ and $\big\{\bigtriangleup_{j}^{r}(k)\big\}_{k\geqslant 0}$ are mutually independent of each other for $i\neq j$ or $q\neq r$; and
\item[(c)] take $\iota(k)=\frac{1}{k^{1-\epsilon}}$ and $c(k)=\frac{1}{k^{\delta}}$ with $0\leqslant\epsilon<\dfrac{1}{4}$, and $ \frac{1}{2}-\epsilon>\delta>\epsilon $  in  randomized difference (\ref{3-6}).
\end{itemize}
\end{condition}

\begin{remark}\label{Rem4}
\begin{itemize}
\item[(a)]The step-size $\iota(k)$ satisfies the following stochastic approximation step-size condition in \cite{chen1999kiefer, chen2006stochastic}:
\begin{align*}
\iota(k)>0, \quad \sum_{k=1}^{\infty}\iota(k)=\infty, \quad \sum_{k=1}^{\infty}\iota^{2}(k)<\infty.
\end{align*}
\item[(b)]$c(k)$ used in randomized difference (\ref{3-6}) satisfies
\begin{align*}
c(k)>0,\quad c(k)\rightarrow 0.
\end{align*}
\item[(c)]The chosen unit parameter $\frac{\iota(k)}{c(k)}$ satisfies:
\begin{align*}
\frac{\iota(k)}{c(k)}>0, \quad  \sum_{k=1}^{\infty}\frac{\iota^{2}(k)}{c^{2}(k)}>0, \quad\sum_{k=1}^{\infty}\iota(k)c(k)<\infty.
\end{align*}
\end{itemize}
\end{remark}

\section{Main results}\label{sec4}
In this section, we first  show that the estimate $(x_{i}(k),\lambda_{i}(k))$ converges to an optimal point $(x^{*},\lambda^{*})$ almost surely
by Algorithm \ref{Alg1}, and then discuss the convergence rate of Algorithm \ref{Alg1}.

\subsection{Convergence}
Denote the transition matrix of $W(k)$ as $\Psi(k,s)=W(k)W(k-1)\cdots W(s), k\geqslant s$, where $\big[\Psi(k,s)\big]_{ij}$ is the $ij$-th element of  $\Psi(k,s)$.
The following result was given in Proposition 1 of \cite{nedic2009distributed}.
\begin{lemma}\label{Lem7}
Under Assumptions \ref{Ass2}, $\Big|\big[\Psi(k,s)\big]_{ij}-\frac{1}{n}\Big|\leqslant\mu\beta^{k-s},\;\forall
k>s$, where $\mu=2\big(1+\eta^{-K_{0}}\big)/\big(1-\eta^{-K_{0}}\big)$,
with $K_{0}=\big(n-1\big)\kappa$ and $\beta=\big(1-\eta^{-K_{0}}\big)^{1/K_{0}}<1$.
\end{lemma}

Here is a theorem regarding convergence analysis of the proposed algorithm.

\begin{theorem}\label{The1}
With  Assumptions \ref{Ass1}-\ref{Ass2},
\begin{itemize}
\item[(a)] all
sequences $\{ \lambda_{i}(k)\},\;i\in\mathcal{N}$ generated by Algorithm \ref{Alg1} reach the same point $\lambda^{*}$ (which  depends on initial parameters $\lambda_{i}(0)'s$).
\item[(b)] all
sequences $\{ x_{i}(k)\},\;i\in\mathcal{N}$ generated by Algorithm    \ref{Alg1} converge to the same optimal
point $x^{*}$ almost surely.
\end{itemize}
\end{theorem}

Before we give the proof of of Theorem \ref{The1}, let us introduce the following three lemmas. The first lemma
 gives an upper bound for the  Euclidean norm of $d_{i}(k)$ in expectation; the second lemma analyzes the consensus in $L_{1}$ norm of estimates $x_{i}(k)$ in  Algorithm \ref{Alg1}; the third lemma analyzes the lower bound of the cross term of $d_{i}(k)$ and $(\xi_{i}(k)-x^*)$ in expectation and in conditional expectation with respect to $F(k)$, where $x^*$ is the optimal solution of \eqref{3-4} for fixed common point $\lambda^{*}$. The proofs of these lemmas are given in Appendix. % the fourth one analyzes that $\{ x_{i}(k)\},\;i\in\mathcal{N}$ converge to the same random variable almost surely and  the last one analyzes that $\{ x_{i}(k)\},\;i\in\mathcal{N}$ converge to $x^{*}$ in $L_{2}$.

 \begin{lemma}\label{Lem8}
With Assumption \ref{Ass1}, following statements hold:
\begin{itemize}
\item[(a)]$\big\|\partial f_{i_{x}}(x,\lambda)\big\|\leqslant L$ and $\big\|\partial  f_{i_{\lambda}}(x,\lambda)\big\|\leqslant K$.
\item[(b)] the first order moment and second moment of $d_{i}(k)$ are bounded by
\begin{align*}
\mathbb{E}\big\|d_{i}(k)\big\|\leqslant nM_{1}M_{2}L, \quad \mathbb{E}\big\|d_{i}(k)\big\|^{2}\leqslant (nM_{1}M_{2}L)^{2}.
\end{align*}
\end{itemize}
$L$ and $K$ are Lipschitz constants with respect to $x$ and $\lambda$ in Lemma \ref{Lem13}.
\end{lemma}

\begin{lemma}\label{Lem9}
With  Assumptions \ref{Ass1}-\ref{Ass2}, the consensus of estimate $x_{i}(k)$ in $L_{1}$ is achieved by Algorithm \ref{Alg1}, that is, for $ i,j=1,2,\ldots,n$, $$\lim_{k\rightarrow \infty}\mathbb{E}\big\|x_{i}(k)-x_{j}(k)\big\|=0.$$
\end{lemma}

\begin{lemma}\label{Lem10}
With Assumption \ref{Ass1}, the cross term of $d_{i}(k)$ and $\xi_{i}(k)-\xi^*$ is lower bounded
\begin{itemize}
\item[(a)]
in conditional expectation with respect to $F(k)$ as follows:
\begin{align*}
&\mathbb{E}\big[\big\langle d_{i}(k),\;x_{i}(k)-\xi^*\big\rangle \big|F(k)\big] \notag\\ \geqslant  &
f_{i}\big(\bar{x}(k),\bar{\lambda}(k)\big)-f_{i}\big(x^*,\lambda^{*}\big)-L\big\|\xi_{i}(k)-\bar{x}(k)\big\|-Bc(k)\notag\\-&K\big\|\lambda_{i}(k)-\bar{\lambda}(k)\big\|-K\|\lambda_{i}(k)-\lambda^{*}\big\|-c(k)L\big\|\bigtriangleup_{i}(k)\big\|,
\end{align*}
\item[(b)]in expectation as follows:
\begin{align*}
&\mathbb{E}\big[\big\langle d_{i}(k),\;\xi_{i}(k)-x^*\big\rangle \big] \notag\\\geqslant &  \mathbb{E}\big[  f_{i}\big(\bar{x}(k),\lambda^{*}\big)-f_{i}\big(x^*,\lambda^{*}\big)\big]-L\mathbb{E}\big\|\xi_{i}(k)-\bar{x}(k)\big\|\notag\\-&2K\mathbb{E}\|\lambda_{i}(k)-\lambda^{*}\big\|-c(k)L\mathbb{E}\big\|\bigtriangleup_{i}(k)\big\|-Bc(k),
\end{align*}
\end{itemize}
where $L$ is the Lipschitz constant with respect to $x$, $K$ is the Lipschitz constant with respect to $\lambda$ given in Lemma \ref{Lem13}, and $B$ is a positive constant.
\end{lemma}

Then it is time to give the proof of Theorem \ref{The1}.
\begin{proof}
\begin{itemize}
\item[(a)] We claim that, for $ i,j=1,2,\ldots,n$,
$$\lim_{k\rightarrow \infty}\big\|\lambda_{i}(k)-\lambda_{j}(k)\big\|=0 \quad a.\; s.$$
Recalling the transition matrix $\Psi(k,s)$ and $\lambda_{i}(k+1)$  in \eqref{3-9}, we have
\begin{align}\label{4-1}
\lambda_{i}(k+1)=\sum_{j=1}^{n} \big[\Psi(k,0)\big]_{ij}\lambda_{j}(0).
\end{align}
Define $\bar{\lambda}(k+1)=\frac{1}{n}\sum_{i=1}^{n}\lambda_{i}(k+1)$. According to Assumption \ref{Ass1} and by an analogous induction,
\begin{align}\label{4-2}
\bar{\lambda}(k+1)=\frac{1}{n}\sum_{i=1}^{n}\lambda_{i}(0).
\end{align}
Therefore, for $i\in \mathcal{N}$,
    \begin{align}\label{4-3}
\big\|\lambda_{i}(k+1)-\bar{\lambda}(k+1)\big\|&\leqslant\sum_{j=1}^{n}\Big |\big [\Psi(k,0)\big]_{ij}-\frac{1}{n}\Big|\big\|\lambda_{j}(0)\big\|.
    \end{align}
Plugging in the estimate of $\Psi(k,s)$ in Lemma \ref{Lem7} leads to
    \begin{align}\label{4-4}
\big\|\lambda_{i}(k+1)-\bar{\lambda}(k+1)\big\|&\leqslant n\delta
\beta^{k}\max_{1\leqslant i\leqslant
n}\big\|\lambda_{i}(0)\big\|.
    \end{align}
Therefore,
\begin{align}\label{4-5}
\lim_{k\rightarrow\infty}\big\|\lambda_{i}(k)-\bar{\lambda}(k)\big\|=0,\;\forall
i\in \mathcal{N}.
\end{align}
\item[(b)]
We claim that, for $ i,j=1,2,\ldots,n$,
$$\lim_{k\rightarrow \infty}\big\|x_{i}(k)-x_{j}(k)\big\|=0 \quad a.\;s.$$
From Lemma \ref{Lem9},
 $\lim_{k\rightarrow\infty}\mathbb{E}\big\|x_{i}(k+1)-\bar{x}(k+1)\big\|=0$ holds. Still
\begin{align}\label{D-1}
 0&\leqslant\mathbb{E}\Big[\underset{k\rightarrow\infty}{\liminf}\big\|x_{i}(k+1)-\bar{x}(k+1)\big\|\Big]\notag\\&\leqslant\underset{k\rightarrow\infty}{\liminf}\mathbb{E}\big\|x_{i}(k+1)
-\bar{x}(k+1)\big\|=0,
 \end{align}
which implies $\mathbb{E}\Big[\underset{k\rightarrow\infty}{\liminf}\big\|x_{i}(k+1)-\bar{x}(k+1)\big\|\Big]=0$. Therefore, $\underset{k\rightarrow\infty}{\liminf} \big\|x_{i}(k+1)-\bar{x}(k+1)\big\|=0$ holds almost surely. Since $\sum_{i=1}^{n}\big\|x_{i}(k+1)-\bar{x}(k+1)\big\|^{2}\leqslant \sum_{i=1}^{n}\big\|x_{i}(k+1)-\bar{x}(k)\big\|^{2}$ by Lemma \ref{Lem2} and $\big\|x_{i}(k+1)-\bar{x}(k)\big\|^2\leqslant\big\|\hat{\xi}_{i}(k)-\bar{x}(k)\big\|^2$ by Lemma \ref{Lem3}, we have
    \begin{align}\label{D-2}
 &\sum_{i=1}^{n}  \big \|x_{i}(k+1)-\bar{x}(k+1)\big\|^{2}\notag\\\leqslant &\sum_{i=1}^{n} \big \|\hat{\xi}_{i}(k)-\bar{x}(k)\big\|^2\notag\\\leqslant & \sum_{i=1}^{n} \sum_{j=1}^{n}w_{ij}(k)\big\|x_{j}(k)-\bar{x}(k)\big\|^2+\iota^2(k)\sum_{i=1}^{n}\big\|d_{i}(k)\big\|^2\notag\\+&2\iota(k) \sum_{i=1}^{n}\big \|d_{i}(k)\big\|\sum_{j=1}^{n}w_{ij}(k)\big\|x_{j}(k)-\bar{x}(k)\big\|.
    \end{align}
According to Assumption \ref{Ass2}(b),
\begin{align}
\sum_{i=1}^{n}\sum_{j=1}^{n}w_{ij}(k)\big\|x_{j}(k)-\bar{x}(k)\big\|^{2}=\sum_{i=1}^{n}\big\|x_{i}(k)-\bar{x}(k)\big\|^{2}.
\end{align}
Taking the conditional expectation of both sides of \eqref{D-2} yields
\begin{align}\label{D-3}
&\sum_{i=1}^{n}\mathbb{E}\Big[\big \|x_{i}(k+1)-\bar{x}(k+1)\big\|^{2}\big|F(k)\Big]
\notag\\ \leqslant  &\sum_{i=1}^{n}\big \|x_{j}(k)-\bar{x}(k)\big\|^2+ \sum_{i=1}^{n}\iota^2(k) \mathbb{E}\big\|d_{i}(k)\big\|^2\notag\\+& \sum_{i=1}^{n}2n\iota(k)  \mathbb{E}\big\|d_{i}(k)\big\|\mathbb{E}\big\|x_{i}(k)-\bar{x}(k)\big\|.
    \end{align}
According to Remark \ref{Rem4} and Lemma \ref{Lem8}(b), $\sum_{k=1}^{\infty}\sum_{i=1}^{n}\iota^2(k) \mathbb{E}\big\|d_{i}(k)\big\|^2<\infty$.  By Theorem 6.2 of \cite{ram2010distributed}, $\sum_{k=1}^{\infty}\iota(k)\big\|x_{i}(k)-\bar{x}(k)\big\|<\infty$ with
probability $1$. From Lemma \ref{Lem8}(a), $\sum_{k=1}^{\infty} \sum_{i=1}^{n}2n\iota(k)  \mathbb{E}\big\|d_{i}(k)\big\|\mathbb{E}\big\|x_{i}(k)-\bar{x}(k)\big\|<\infty$. Therefore, $\lim_{k\rightarrow\infty}\big\|x_{i}(k)-\bar{x}(k)\big\|=0$ almost surely by Lemma
\ref{Lem4}.
\item[(c)] Clearly, $\big\|x_{i}(k+1)-x^*\big\|^2\leqslant \big\|\hat{\xi}_{i}(k)-x^{*}\big\|^2$ according to Lemma \ref{Lem3}. Then
\begin{align}\label{D-4}
\big\|x_{i}(k+1)-x^*\big\|^2&\leqslant
\big\|\xi_{i}(k)-x^* \big\|^2 +
\iota^2(k)
\big\|d_{i}(k)\big\|^2\notag\\&-2\iota(k)\big\langle d_{i}(k),\;\xi_{i}(k)-x^*\big\rangle.
\end{align}
Taking conditional expectation on both sides of \eqref{D-4} gives
   \begin{align}\label{D-5}
  &  \mathbb{E}\Big[\big\|x_{i}(k+1)-x^*\big\|^{2}\big|F(k)\Big] \notag\\\leqslant &\mathbb{E}\Big[\big\|\xi_{i}(k)-x^{*}\big \|^{2} \big|F(k)\Big]+\iota^{2}(k)\mathbb{E}\Big[\big\|d_{i}(k)\big\|^2\big|F(k)\Big]\notag\\-&2\iota(k)\mathbb{E}\Big[\big\langle d_{i}(k),\; \xi_{i}(k)-x^*\big\rangle \big|F(k)\Big]
    \end{align}
for all $k=0,1,2,\ldots$. By the double stochasticity of matrix $W(k)$ in Assumption \ref{Ass2}(b),
    \begin{align}\label{D-6}
    \sum_{i=1}^{n} \mathbb{E}\Big[\big\|\xi_{i}(k)-x^{*}\big\|^{2} \big|F(k)\Big]
    &\leqslant \sum_{i=1}^{n} \big\|x_{i}(k)-x^{*}\big\|^2,\notag\\
   \sum_{i=1}^{n} \mathbb{E}\Big[\big\|\xi_{i}(k)-\bar{x}(k)\big\| \big|F(k)\Big]&\leqslant \sum_{i=1}^{n} \big\|x_{i}(k)-\bar{x}(k)\big\|.
    \end{align}
Then, with probability $1$, for $i\in \mathcal{N}$, it holds
    \begin{align}\label{D-7}
   & \sum_{i=1}^{n}\mathbb{E}\Big[\big\|x_{i}(k+1)-x^* \big\|^2\big|F(k)\Big]\notag\\\leqslant&\sum_{i=1}^{n}\Big[\big \|x_{i}(k)-x^*\big\|^2+\big[O_{i}(k)\big]_{1}+\big[O_{i}(k)\big]_{2}+\big[O_{i}(k)\big]_{3}\notag\\+&\big[O_{i}(k)\big]_{4}+\big[O_{i}(k)\big]_{5}+\big[O_{i}(k)\big]_{6}-J_{i}(k)\Big],
    \end{align}
     where
     \begin{align*}
     \begin{cases}
       \big[O_{i}(k)\big]_{1}&=\iota^{2}(k)\mathbb{E}\Big[\big\|d_{i}(k)\big\|^2\big|F(k)\Big]\\
      \big[O_{i}(k)\big]_{2}&=2\iota(k)L\mathbb{E}\big\|x_{i}(k)-\bar{x}(k)\big\|\\
      \big[O_{i}(k)\big]_{3}&=4\iota(k)c(k)L\mathbb{E}\big\|\bigtriangleup(k)_{i}\big\|
      \\\big[O_{i}(k)\big]_{4}&=2\iota(k)L\mathbb{E}\big\|\lambda_{i}(k)-\bar{\lambda}(k)\big\|   \\
      \big[O_{i}(k)\big]_{5}&=2\iota(k)L\mathbb{E}\big\|\lambda_{i}(k)-\lambda^{*}\big\|  \\
        \big[O_{i}(k)\big]_{6}&=2B\iota(k)c(k)\\
      J_{i}(k)&=2\iota(k)\big[f_{i}\big(\bar{x}(k),\bar{\lambda}(k)\big)-f_{i}\big(x^*,\lambda^{*}\big)\big].
     \end{cases}
     \end{align*}
Recalling Remark \ref{Rem4} and Lemma \ref{Lem8}, $\sum_{k=1}^{\infty}\big[O_{i}(k)\big]_{1}<\infty$.
By the proof in part (a),  $\sum_{k=1}^{\infty}\big[O_{i}(k)\big]_{2}<\infty$. $\sum_{k=1}^{\infty}\big[O_{i}(k)\big]_{3}<\infty$. From Theorem \ref{The1},
$\sum_{k=1}^{\infty}\big[O_{i}(k)\big]_{4}<\infty$ and  $\sum_{k=1}^{\infty}\big[O_{i}(k)\big]_{5}<\infty$. With Remark \ref{Rem4}, $\sum_{k=1}^{\infty}\big[O_{i}(k)\big]_{6}<\infty$. Therefore, $\sum_{k=1}^{\infty}\sum_{i=1}^{n}\Big[\big[O_{i}(k)\big]_{1}+\big[O_{i}(k)\big]_{2}+\big[O_{i}(k)\big]_{3}+\big[O_{i}(k)\big]_{4}+\big[O_{i}(k)\big]_{5}+\big[O_{i}(k)\big]_{6}\Big]<\infty$. From Lemma \ref{Lem7}, $\sum_{i=1}^{n} \big\|x_{i}(k)-x^*\big\|^{2}$ converges almost surely with $\sum_{k=1}^{\infty}\sum_{i=1}^{n}J_{i}(k)<\infty$.  Since $\sum_{i=1}^{\infty}\iota(k)=\infty$,
$$\lim \inf_{k\rightarrow\infty} f_{i}\big(\bar{x}(k),\bar{\lambda}(k)\big)= f_{i}\big(x^*,\lambda^{*}\big)$$
holds almost surely, Therefore, the sequence $\lim_{k\rightarrow\infty}\sum_{i=1}^{n} \big\|\xi_{i}(k)-\xi^*\big\|^{2}=0$ with probability 1.
The proof is completed.
\end{itemize}
\end{proof}

\begin{remark}
Most of existing zeroth-order (distributed) algorithms \cite{anit2018distributed, hajinezhad2017zeroth, yuan2015randomized, yuan2015zeroth} are based on the assumption that (local) objective functions are smooth. However, in our article, we assume that  local interval-valued objective functions are nonsmooth. In this case, a direct application of the subgradient and the step-size selection for most of the existing distributed first-order algorithms  \cite{nedic2009distributed, yi2015distributed, nedic2010constrained, ram2010distributed, zhang2015distributed, cherukuri2015distributed} are not applicable. In the proof of Theorem \ref{The1}, we select a different parameter $\frac{\iota(k)}{c(k)}$, which guarantees the application of supermartingale convergence theorem in \cite{polyak1987introduction}. Also, we make  use of Lebourg's mean value theorem \cite{clarke1998nonsmooth} to estimate local subgradient information.
\end{remark}
\subsection{Convergence rate}
We further analyze the convergence rate of Algorithm \ref{Alg1}.  Denote $(x^*,\lambda^{*}\big)$ as the optimal solution of problem \eqref{3-4}, where $\lambda^{*}$ is given in Theorem \ref{The1} and $x^{*}\in \arg \min_{x_{i}=x_{i}\in X}f(x,\lambda^{*})$.  Here is the main result.

\begin{theorem}\label{The3}
With Assumptions \ref{Ass1}-\ref{Ass2}, for Algorithm \ref{Alg1}, we have
\begin{align*}
R(T) &\sim O\Big(\dfrac{1}{T^{\epsilon}}\Big).
\end{align*}
\end{theorem}

\begin{proof}
By taking expectation to both sides of \eqref{D-4}, we obtain
\begin{align}\label{F-1}
    \mathbb{E}\big\|x_{i}(k+1)-x^*\big\|^{2} &\leqslant \mathbb{E}\big\|\xi_{i}(k)-x^{*} \big\|^{2}+\iota^{2}(k)\mathbb{E}\big\|d_{i}(k)\big\|^2\notag\\&-2\iota(k)\mathbb{E}\big[\big\langle d_{i}(k),\; \xi_{i}(k)-x^*\big\rangle\big].
\end{align}

By the double stochasticity of matrix $W(k)$ given in Assumption \ref{Ass2}(b), we obtain
\begin{align}
\sum_{i=1}^{n} \mathbb{E}\big\|\xi_{i}(k)-x^{*}\big\|^{2} &=   \sum_{i=1}^{n} \mathbb{E}\Big\|\sum_{j=1}^{n}w_{ij}(k)x_{j}(k)-x^{*}\Big\|^{2}\notag\\&\leqslant \sum_{i=1}^{n} \mathbb{E}\big\|x_{i}(k)-x^{*}\big\|^2,\label{F-2}\\\sum_{i=1}^{n} \mathbb{E}\big\|\xi_{i}(k)-\bar{x}(k)\big\|&= \sum_{i=1}^{n} \mathbb{E}\Big\|\sum_{j=1}^{n}w_{ij}(k)x_{j}(k)-\bar{x}(k)\Big\|\notag\\&\leqslant \sum_{i=1}^{n}\mathbb{E} \big\|x_{i}(k)-\bar{x}(k)\big\|.\label{F-3}
\end{align}
By \eqref{F-2}, \eqref{F-3} and Lemma \ref{Lem10}, we have
\begin{align}\label{F-4}
&\sum_{k=1}^{T}\sum_{i=1}^{n}\mathbb{E} \big\|x_{i}(k+1)-x^{*}\big\|^{2}\notag\\\leqslant &\sum_{k=1}^{T}\sum_{i=1}^{n}\mathbb{E} \big\|x_{i}(k)-x^{*}\big\|^{2}+4K\sum_{k=1}^{T}\sum_{i=1}^{n}\iota(s)\mathbb{E}\big\|\lambda_{i}(k)-\lambda^{*}\big\|
\notag\\+&2L\sum_{k=1}^{T}\sum_{i=1}^{n}\iota(k)\mathbb{E}\big\|x_{i}(k)-\bar{x}(k)\big\|+2nB\sum_{k=1}^{T}\iota(k)c(k)\notag\\+&4L\sum_{k=1}^{T}\sum_{i=1}^{n}\iota(k)c(k)\mathbb{E}\big\|\bigtriangleup_{i}(k)\big\|+\sum_{k=1}^{T}\sum_{i=1}^{n}\iota^{2}(k)\mathbb{E}\big\|d_{i}(k)\big\|^2
\notag\\-&2\sum_{k=1}^{T}\sum_{i=1}^{n}\iota(k)\mathbb{E}\Big[f_{i}\big(\bar{x}(k),\lambda^{*} \big)-f_{i}\big(x^*,\lambda^{*}\big)\Big].
\end{align}
Therefore, by taking summation of both sides of \eqref{F-1} for $k=1,2,\ldots T$ and $i=1,2,\ldots n$, we get
\begin{align}\label{F-5}
&\sum_{k=1}^{T}\sum_{i=1}^{n}\mathbb{E}\Big[f_{i}\big(x_{i}(k),\lambda_{i}(k) \big)-f_{i}\big(x^*,\lambda^{*}\big)\Big]
\notag\\\leqslant &\sum_{k=1}^{T}\frac{1}{\iota(k)}\bigg[ \sum_{i=1}^{n}\mathbb{E} \big\|x_{i}(k+1)-x^{*}\big\|^{2}- \sum_{i=1}^{n}\mathbb{E} \big\|x_{i}(k)-x^{*}\big\|^{2}\bigg]\notag\\+&3K\sum_{k=1}^{T}\sum_{i=1}^{n}\mathbb{E}\big\|\lambda_{i}(k)-\lambda^{*}\big\|
+2L\sum_{k=1}^{T}\sum_{i=1}^{n}\mathbb{E}\big\|x_{i}(k)-\bar{x}(k)\big\|\notag\\+&nB\sum_{k=1}^{T}c(k)+2L\sum_{k=1}^{T}\sum_{i=1}^{n}c(k)\mathbb{E}\big\|\bigtriangleup_{i}(k)\big\|\notag\\+&\frac{1}{2}\sum_{k=1}^{T}\sum_{i=1}^{n}\iota(k)\mathbb{E}\big\|d_{i}(k)\big\|^2
\end{align}
Note that $\iota(k)=\frac{1}{k^{1-\epsilon}}$ and $c(k)=\frac{1}{k^{\delta}}$, $0\leqslant\epsilon<\dfrac{1}{4}$, and $ \frac{1}{2}-\epsilon>\delta>\epsilon $. Since $X$ is bounded in $\mathcal{R}^{m}$, for $x\in X$, there exists a constant $M_{x}$ such that $\big\|x\big\|\leqslant M_{x}$. For the first term on the right hand side of \eqref{F-5}, we have
\begin{align}\label{F-6}
&\sum_{k=1}^{T}\frac{1}{\iota(k)}\bigg[ \sum_{i=1}^{n}\mathbb{E} \big\|x_{i}(k+1)-x^{*}\big\|^{2}- \sum_{i=1}^{n}\mathbb{E} \big\|x_{i}(k)-x^{*}\big\|^{2}\bigg]
\notag\\\leqslant&\dfrac{M_{1}}{\iota_{T}}.
\end{align}
By Lemmas \ref{Lem7} and \ref{Lem8}, for the second term and third term on the right hand side of \eqref{F-5}, we have
\begin{align}
2L\sum_{k=1}^{T}\sum_{i=1}^{n}\mathbb{E}\big\|x_{i}(s)-\bar{x}(s)\big\|&\leqslant M_{21}\sum_{k=1}^{T}\iota(s)\leqslant M_{21}T^{\epsilon},\label{F-7}\\
3K\sum_{k=1}^{T}\sum_{i=1}^{n}\mathbb{E}\big\|\lambda_{i}(s)-\lambda^{*}\big\|&\leqslant M_{22}\sum_{k=1}^{T}\iota(s)\leqslant M_{22}T^{\epsilon}.\label{F-8}
\end{align}
Clearly, for  the fourth term and fifth term on the right hand side of \eqref{F-5}, we have
\begin{align}\label{F-9}
nB\sum_{k=1}^{T}c(s)\leqslant M_{31}T^{1-\delta},
\end{align}
and
\begin{align}\label{F-10}
2L\sum_{k=1}^{T}\sum_{i=1}^{n}c(s)\mathbb{E}\big\|\bigtriangleup_{i}(s)\big\|\leqslant  M_{32}T^{1-\delta}.
\end{align}
For the last term on the right hand side of \eqref{F-5}, we have
\begin{align}\label{F-11}
\dfrac{1}{2}\sum_{k=1}^{T}\sum_{i=1}^{n}\iota(s)\mathbb{E}\big\|d_{i}(s)\big\|^2\leqslant M_{23}T^{\epsilon}.
\end{align}
Thus, the conclusion follows with $M_{31}+M_{32}=M_{3}$ and $M_{21}+M_{22}+M_{23}=M_{2}$.
Therefore,
 \begin{align}\label{F-12}
&\sum_{k=1}^{T}\sum_{i=1}^{n}\mathbb{E}\Big[f_{i}\big(x_{i}(k),\lambda_{i}(k)\big)-f_{i}\big(x^*,\lambda^{*}\big)\Big] \notag\\\leqslant &M_{1}T^{1-\epsilon}+M_{2}T^{\epsilon}+M_{3}T^{1-\delta},
\end{align}
where $M_{1}$, $M_{2}$ and $M_{3}$ are constants. Dividing both sides of \eqref{F-12} by $T$ gives
\begin{align}\label{F-13}
&\frac{1}{T}\sum_{k=1}^{T}\sum_{i=1}^{n}\mathbb{E}\Big[f_{i}\big(x_{i}(k),\lambda_{i}(k)\big)-f_{i}\big(x^*,\lambda^{*}\big)\Big]\notag\\ \sim &O\Big(\max\Big\{\frac{1}{T^{\epsilon}}, \frac{1}{T^{\delta}}, \frac{1}{T^{1-\epsilon}}\Big\}\Big).
\end{align}
The proof is completed.
\end{proof}

\begin{remark}
The convergence rate in Theorem \ref{The3} is also corresponding to the regret bound, defined as
\begin{align*}
R(T)=\frac{1}{T}\sum_{k=1}^{T}\sum_{i=1}^{n}\mathbb{E}\Big[f_{i}\big(x_{i}(k),\lambda_{i}(k)\big)-f_{i}\big(x^*,\lambda^{*}\big)\Big],
\end{align*}
(as given in online optimization \cite{hazan2016introduction}) for the following interval optimization problem:
\begin{align}
\min_{x_i=x_j\in X} \quad F(x,\lambda(k))=\sum_{i=1}^{n}f_{i}(x_i,\lambda_{i}(k)),\quad T>0.
\end{align}
%\begin{align}
%\min \quad F(x,\lambda)=\sum_{k=1}^{T}\sum_{i=1}^{n}f_{i}(x_{i}(k),\lambda_{i}(k)),\quad T>0.
%\end{align}
Also, the established convergence rate $O(\frac{1}{T^{\epsilon}})$ is the best convergence rate for distributed
zeroth-order convex optimization. Note that this convergence rate is slower than that of distributed first-order methods \cite{ nedic2010constrained, ram2010distributed, zeng2017distributed} for the limitation of parameter choices and the prior function knowledge.
\end{remark}

\section{Simulation}\label{sec5}

In this section, we demonstrate simulations of the distributed stochastic zeroth-order algorithm for the following distributed interval-valued quadratic problem:
\begin{align*}
\min \quad  G(\bm{x})&=\sum_{i=1}^{5}[\upsilon_{1i},\upsilon_{2i}]\|x-\rho_{i}\|^{2}\\
s.\; t.\; \;\quad x&\in X,
\end{align*}
where $\upsilon_{1i}$, $\upsilon_{1i}\in \mathcal{R}$ and $\rho_{i}\in \mathcal{R}^{p}$. This problem is motivated from centralized quadratic interval-valued learning \cite{bhurjee2012efficient} and distributed optimization \cite{nedic2016stochastic}.

Take $X=\big\{x\big|\|x\|\leqslant 100\big\}$, $x\in \mathcal{R}$ with $[\upsilon_{1i},\upsilon_{2i}]=[0.5, 2]$. Take  $\rho_{1}=3$ , $\rho_{2}=2$, $\rho_{3}=1$, $\rho_{4}=0$, $\rho_{5}=-1$.   Then we consider parameters in the proposed algorithm by setting the step-size $\iota(k)=\frac{1}{k^{\frac{7}{8}}}$ and $c(k)=\frac{1}{k^{\frac{1}{4}}}$ used in randomized differences, along with $\lambda_{1}(0)=0.1$, $\lambda_{2}(0)=0.3$, $\lambda_{3}(0)=0.5$, $\lambda_{4}(0)=0.7$, $\lambda_{5}(0)=0.9$ and $x_{i}(0)$'s $=0$.

Then we investigate the convergence performance of the distributed stochastic zeroth-order  algorithm. Simulation results are based on a $5$-agent time-varying network, whose communication topology between agents can be described by Fig. 2. Also, Figs. \ref{fig3} and   \ref{fig4} show the convergence performance of the proposed algorithm. We can get a Pareto solution as $(0.500,0.996)$ for $500$ iterations.
\begin{figure}[!htbp]\label{fig2}
\centering
\begin{tikzpicture}[line cap=round,line join=round,>=triangle 45,x=0.8cm,y=0.8cm]
\clip(-3,-3) rectangle (2,2);
\fill[line width=0.8pt,color=ttttff,fill=ttttff,fill opacity=0.10] (-1.,-2.) -- (1.,-2.) -- (1.618,-0.097) -- (0.,1.08) -- (-1.618,-0.097) -- cycle;
\draw [line width=0.4pt,color=qqqqcc] (-1.,-2.)-- (1.,-2.);
\draw [line width=0.4pt,color=qqqqcc] (1.,-2.)-- (1.618,-0.097);
\draw [line width=0.4pt,color=qqqqcc] (1.618,-0.09)-- (0.,1.08);
\draw [line width=0.4pt,color=qqqqcc] (0.,1.08)-- (-1.618,-0.097);
\draw [line width=0.4pt,color=qqqqcc] (-1.618,-0.097)-- (-1.,-2.);
\draw [line width=0.4pt,color=qqqqcc] (-1.618,-0.097)-- (1.618,-0.097);
\draw [line width=0.4pt,color=qqqqcc] (0.,1.08)-- (1.,-2.);
\begin{scriptsize}
\draw [fill=xdxdff] (-1.,-2.) circle (4.5pt);
\draw[color=xdxdff] (-0.86,-1.47) node {5};
\draw [fill=xdxdff] (1.,-2.) circle (4.5pt);
\draw[color=xdxdff] (1.14,-1.47) node {4};
\draw [fill=xdxdff] (1.618,-0.097) circle (4.5pt);
\draw[color=xdxdff] (0.14, -1) node {(a)};
\draw[color=xdxdff] (1.76,0.43) node {3};
\draw [fill=xdxdff] (0.,1.08) circle (4.5pt);
\draw[color=xdxdff] (0.14,1.61) node {2};
\draw [fill=xdxdff] (-1.618,-0.097) circle (4.5pt);
\draw[color=xdxdff] (-1.48,0.43) node {1};
\end{scriptsize}
\end{tikzpicture}
\begin{tikzpicture}[line cap=round,line join=round,>=triangle 45,x=0.8cm,y=0.8cm]
\clip(-3,-3) rectangle (2,2);
\fill[line width=0.8pt,color=ttttff,fill=ttttff,fill opacity=0.10] (1.618,-0.097) -- (0.,1.08) -- (-1.618,-0.097) -- cycle;
\draw [line width=0.4pt,color=qqqqcc] (1.618,-0.09)-- (0.,1.08);
\draw [line width=0.4pt,color=qqqqcc] (0.,1.08)-- (-1.618,-0.097);
\draw [line width=0.4pt,color=qqqqcc] (-1.618,-0.097)-- (1.618,-0.097);
\begin{scriptsize}
\draw [fill=xdxdff] (-1.,-2.) circle (4.5pt);
\draw[color=xdxdff] (-0.86,-1.47) node {5};
\draw [fill=xdxdff] (1.,-2.) circle (4.5pt);
\draw[color=xdxdff] (1.14,-1.47) node {4};
\draw [fill=xdxdff] (1.618,-0.097) circle (4.5pt);
\draw[color=xdxdff] (1.76,0.43) node {3};
\draw[color=xdxdff] (0.14, -1) node {(b)};
\draw [fill=xdxdff] (0.,1.08) circle (4.5pt);
\draw[color=xdxdff] (0.14,1.61) node {2};
\draw [fill=xdxdff] (-1.618,-0.097) circle (4.5pt);
\draw[color=xdxdff] (-1.48,0.43) node {1};
\end{scriptsize}
\end{tikzpicture}

\begin{tikzpicture}[line cap=round,line join=round,>=triangle 45,x=0.8cm,y=0.8cm]
\clip(-3,-3) rectangle (2,2);
\fill[line width=0.8pt,color=ttttff,fill=ttttff,fill opacity=0.10](1.,-2.) -- (1.618,-0.097) -- (0.,1.08)-- cycle;
\draw [line width=0.4pt,color=qqqqcc] (1.,-2.)-- (1.618,-0.097);
\draw [line width=0.4pt,color=qqqqcc] (1.618,-0.09)-- (0.,1.08);
\draw [line width=0.4pt,color=qqqqcc] (0.,1.08)-- (1.,-2.);
\begin{scriptsize}
\draw [fill=xdxdff] (-1.,-2.) circle (4.5pt);
\draw[color=xdxdff] (-0.86,-1.47) node {5};
\draw [fill=xdxdff] (1.,-2.) circle (4.5pt);
\draw[color=xdxdff] (1.14,-1.47) node {4};
\draw [fill=xdxdff] (1.618,-0.097) circle (4.5pt);
\draw[color=xdxdff] (1.76,0.43) node {3};
\draw[color=xdxdff] (0.14, -1) node {(c)};
\draw [fill=xdxdff] (0.,1.08) circle (4.5pt);
\draw[color=xdxdff] (0.14,1.61) node {2};
\draw [fill=xdxdff] (-1.618,-0.097) circle (4.5pt);
\draw[color=xdxdff] (-1.48,0.43) node {1};
\end{scriptsize}
\end{tikzpicture}
\begin{tikzpicture}[line cap=round,line join=round,>=triangle 45,x=0.8cm,y=0.8cm]
\clip(-3,-3) rectangle (2,2);
\fill[line width=0.8pt,color=ttttff,fill=ttttff,fill opacity=0.10] (-1.,-2.) -- (1.,-2.) --(-1.618,-0.097) ;
\draw [line width=0.4pt,color=qqqqcc] (-1.,-2.)-- (1.,-2.);
\draw [line width=0.4pt,color=qqqqcc] (-1.618,-0.097)-- (-1.,-2.);
\begin{scriptsize}
\draw [fill=xdxdff] (-1.,-2.) circle (4.5pt);
\draw[color=xdxdff] (-0.86,-1.47) node {5};
\draw [fill=xdxdff] (1.,-2.) circle (4.5pt);
\draw[color=xdxdff] (1.14,-1.47) node {4};
\draw [fill=xdxdff] (1.618,-0.097) circle (4.5pt);
\draw[color=xdxdff] (1.76,0.43) node {3};
\draw[color=xdxdff] (0.14, -1) node {(d)};
\draw [fill=xdxdff] (0.,1.08) circle (4.5pt);
\draw[color=xdxdff] (0.14,1.61) node {2};
\draw [fill=xdxdff] (-1.618,-0.097) circle (4.5pt);
\draw[color=xdxdff] (-1.48,0.43) node {1};
\end{scriptsize}
\end{tikzpicture}
%\captionsetup{justification=centering}
%\captionsetup{font={footnotesize}}
\caption{Topology of the $5$-agent network. }
\end{figure}
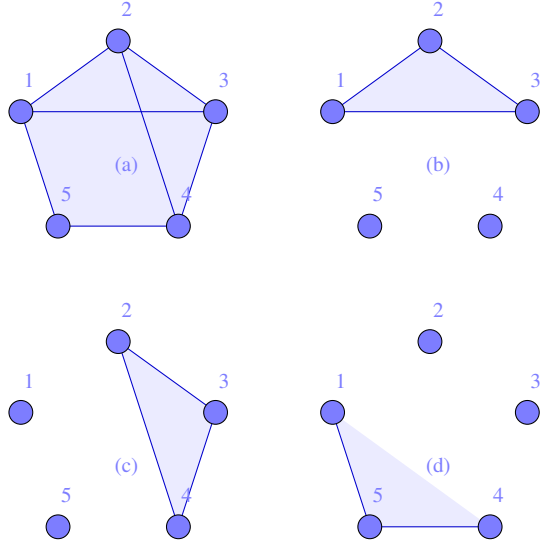

\begin{figure}[ht]
	\centering
	\includegraphics[width=\linewidth]{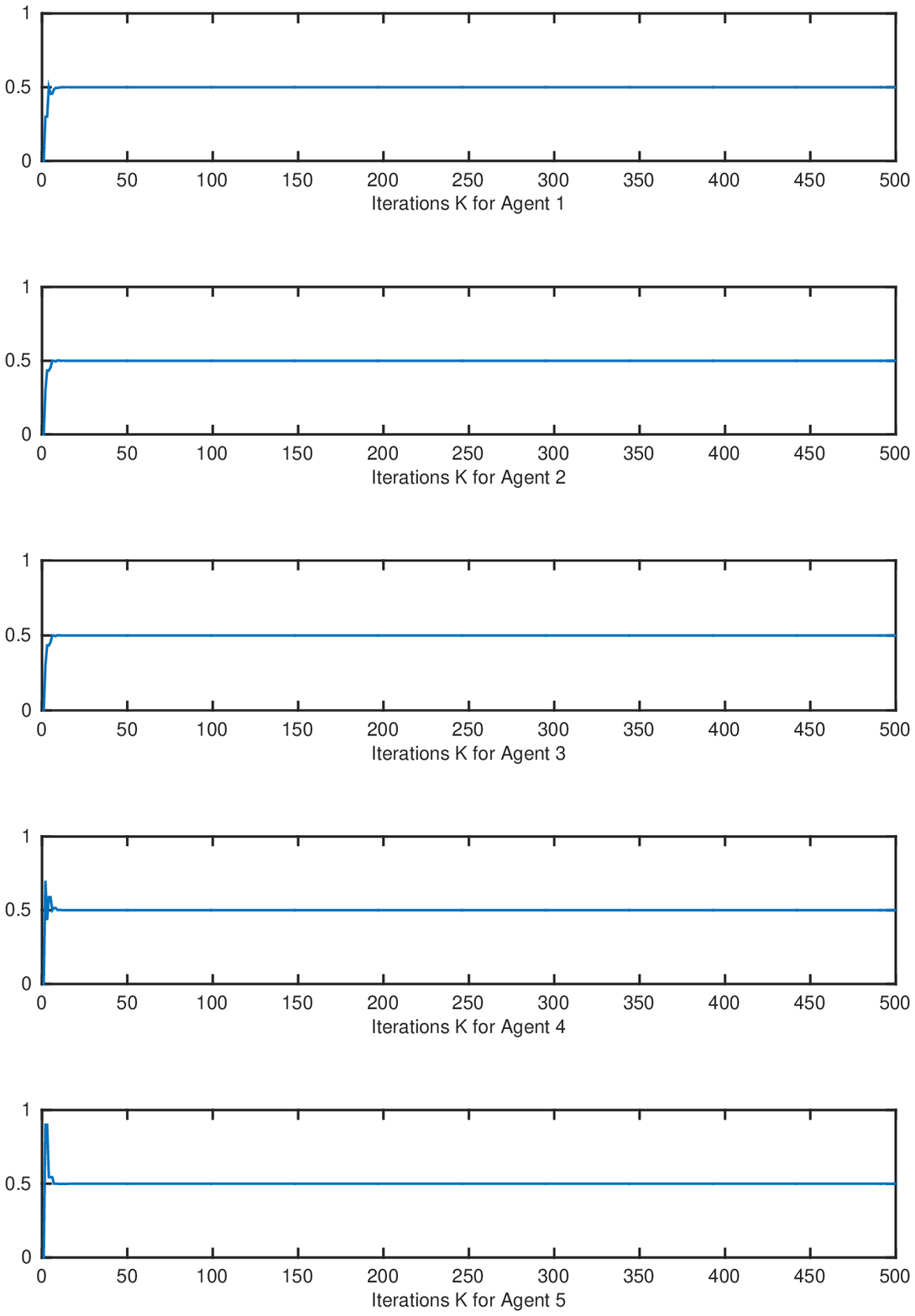}
	\caption{$\lambda_{i}(k)$ for agent $i$}
	\label{fig3}
\end{figure}
\begin{figure}[ht]
	\centering
	\includegraphics[width=\linewidth]{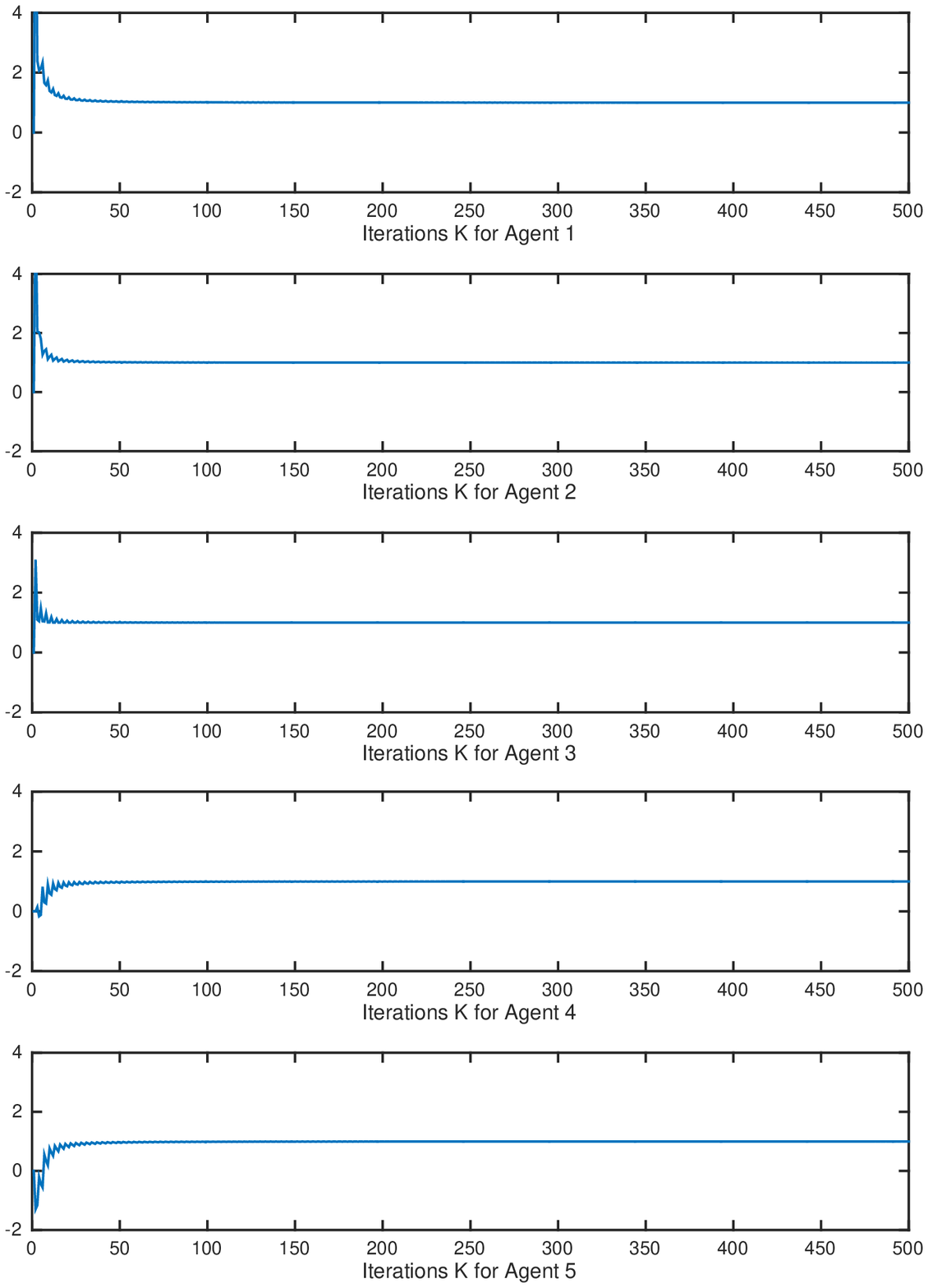}
	\caption{$x_{i}(k)$ for agent $i$}
	\label{fig4}
\end{figure}

\section{Conclusion}\label{sec6}

This paper investigated the distributed interval optimization problem, subject to local convex constraints. The objective functions are compact, interval-valued functions and the network for the distributed design is time-varying. Based on randomization technique, a distributed zeroth-order methodology was developed to find a Pareto optimal solution of distributed interval optimization problem. Moreover, we proved the convergence to a Pareto optimal solution with probability one over time-varying network, and finally gave a numerical example to illustrate the effectiveness of the proposed algorithm.

\appendices
\section*{Acknowledgment}
This work was supported by National Key Research and Development Program of
China (2016YF-B0901900) and NSFC (61733018, 61573344, 61573345, 61603378).

\section{Proof of Lemma \ref{Lem8}}
\noindent (a). Suppose that there is a vector $x$ such that we can choose a subgradient $\triangledown f_{i_{x}}(x,\lambda)\in \partial f_{i_{x}}(x,\lambda)$ with $\|\triangledown f_{i_{x}}(x,\lambda)\|>L$. Suppose $y=x+\triangledown f_{i_{x}}(x,\lambda)$.  Recalling Definition \ref{Def1} gives
 \begin{align*}
 f_{i}(y,\lambda)-f_{i}(x,\lambda)&\geqslant\langle \triangledown f_{i_{x}}(x,\lambda), y-x \rangle\\&\geqslant \big\|\triangledown f_{i_{x}}(x,\lambda)\big\|^{2}>L\big\|\triangledown f_{i_{x}}(x,\lambda)\big\|\\&>L\big\|y-x\big\|,
 \end{align*}
 which contradicts the Lipschitz continuity of $f_{i}\big (x,\lambda\big)$ with respect to $x$. By an analogous proof, $\|\partial f_{i_{\lambda}}(x,\lambda)\|\leqslant K$.

\noindent (b). For $d_{i}(k)$ in \eqref{3-6},
\begin{align}\label{A-1}
d_{i}(k)=\dfrac{\big[y_{i}^{+}(k)-y_{i}^{-}(k)\big]\bigtriangleup_{i}^{-}(k)}{2c(k)}
\end{align}
where $\|y_{i}^{+}(k)-y_{i}^{-}(k)\|=\big\|f_{i}\big(\xi_{i}(k)+c(k)\bigtriangleup_{i}(k),\lambda_{i}(k)\big)-f_{i}\big(\xi_{i}(k)-c(k)\bigtriangleup_{i}(k), \lambda_{i}(k)\big\|\leqslant 2Lc(k)\big\|\bigtriangleup_{i}(k)\big\|$ by Lemma \ref{Lem13}.
Due to Condition \ref{Con1}(a), we have
\begin{align}\label{A-2}
\mathbb{E}\Bigg\|\dfrac{\big[y_{i}^{+}(k)-y_{i}^{-}(k)\big]\bigtriangleup_{i}^{-}(k)}{2c(k)}\Bigg\|\leqslant nM_{1}M_{2}L,
\end{align}
and
\begin{align}\label{A-3}
\mathbb{E}\Bigg\|\dfrac{\big[y_{i}^{+}(k)-y_{i}^{-}(k)\big]\bigtriangleup_{i}^{-}(k)}{2c(k)}\Bigg\|^{2}\leqslant (nM_{1}M_{2}L)^{2}.
\end{align}

\section{Proof of Lemma \ref{Lem9}}
Define, for $i\in \mathcal{N}$ and $k\geqslant 0$
$$
p_{i}(k+1)=x_{i}(k+1)-\xi_{i}(k)=x_{i}(k+1)-\sum_{j=1}^{n}w_{ij}(k)x_{j}(k)
$$
as the error between $x_{i}(k+1)$ and $\xi_{i}(k)$.   From Lemma \ref{Lem5}(b) and the fact that $X$ is a closed convex set, we get
 \begin{align}\label{B-2}
 & \big  \|p_{i}(k+1)\big\|\notag\\=&\bigg\|P_{X}\big(\sum_{j=1}^{n}w_{ij}(k)x_{j}(k)-\iota(k)d_{i}(k)\big)-\sum_{j=1}^{n}w_{ij}(k)x_{j}(k)\bigg\|\notag\\\leqslant &\iota(k)\big\|d_{i}(k)\big\|.
    \end{align}
Rewrite \eqref{3-9} compactly in terms of  $\Psi(k,s)$ and the definition of $p_{i}(k+1)$ as follows:
\begin{align}\label{B-3}
x_{i}(k+1)&=\sum_{j=1}^{n} \big[\Psi(k,0)\big]_{ij}x_{j}(0)+p_{i}(k+1)\notag\\&+\sum_{s=1}^{k}\sum_{j=1}^{n} \big[\Psi(k,s)\big]_{ij}p_{j}(s),
\end{align}
for $k\geqslant s$. Define $\bar{x}(k+1)=\frac{1}{n}\sum_{i=1}^{n}x_{i}(k+1)$. Moreover, with Assumption \ref{Ass1}(b), the following can be obtained similarly:
\begin{align}\label{B-4}
\bar{x}(k+1)=\frac{1}{n}\sum_{i=1}^{n}x_{i}(0)+\frac{1}{n}\sum_{s=1}^{k+1}\sum_{j=1}^{n}p_{j}(s)
\end{align}
Therefore, $\forall i\in \mathcal{N}$,
    \begin{align}\label{B-5}
\big\|x_{i}(k+1)-\bar{x}(k+1)\big\|&\leqslant\sum_{j=1}^{n}\Big |\big [\Psi(k,0)\big]_{ij}-\frac{1}{n}\Big|\big\|x_{j}(0)\big\|\notag\\&+\big\|p_{i}(k+1)\big\|
+\frac{1}{n}\sum_{j=1}^{n}\big\|p_{j}(k+1)\big\|\notag\\&+\sum_{s=1}^{k}\sum_{j=1}^{n} \Big| \big[\Psi(k,s)\big]_{ij}-\frac{1}{n}\Big|\big\|p_{j}(s)\big\|.
    \end{align}
Taking the expectation of \eqref{B-5} yields
 \begin{align}\label{B-6}
&\mathbb{E} \big\|x_{i}(k+1)-\bar{x}(k+1)\big\|
\notag\\\leqslant &\sum_{j=1}^{n}\Big |\big [\Psi(k,0)\big]_{ij}-\frac{1}{n}\Big|\big\|x_{j}(0)\big\|+\mathbb{E}\big\|p_{i}(k+1)\big\|
\notag\\+&\frac{1}{n}\sum_{j=1}^{n}\mathbb{E}\big\|p_{j}(k+1)\big\|+\sum_{s=1}^{k}\sum_{j=1}^{n} \Big| \big[\Psi(k,s)\big]_{ij}-\frac{1}{n}\Big|\mathbb{E}\big\|p_{j}(s)\big\|.
 \end{align}
Plugging in the estimate of $\Psi(k,s)$ in Lemma \ref{Lem7} and the estimate of $p_{i}(k+1)$ in \eqref{B-2}, we have
    \begin{align}\label{B-7}
&\mathbb{E}\big\|x_{i}(k+1)-\bar{x}(k+1)\big\|\notag\\\leqslant &n\delta\beta^{k}\max_{1\leqslant i\leqslant
n}\big\|x_{i}(0)\big\|+\iota(k)\mathbb{E}\big\|d_{i}(k)\big\|+\frac{\iota(k)}{n}\sum_{i=1}^{n}\mathbb{E}\big\|d_{i}(k)\big\|
\notag\\+&\delta\sum_{s=1}^{k}\beta^{k-s}\sum_{i=1}^{n}\iota(s-1)\mathbb{E}\big\|d_{i}(s-1)\big\|.
    \end{align}
From Lemma \ref{Lem8}, $\mathbb{E}\big\|d_{i}(k)\big\|\leqslant L$.  Therefore,
\begin{align}\label{B-8}
\mathbb{E}\big\|x_{i}(k+1)-\bar{x}(k+1)\big\| &\leqslant n\delta\beta^{k}\max_{1\leqslant i\leqslant
n}\big\|x_{i}(0)\big\|\notag\\&\hspace{-2cm}+2\iota(k)nM_{1}M_{2}L\notag\\&\hspace{-2cm}+\delta n^{2}M_{1}M_{2}\sum_{s=1}^{k}\iota(s-1)\beta^{k-s}L.
\end{align}
Since $\sum_{k=1}^{\infty}\iota(k)^{2}<\infty$ with Remark \ref{Rem4}(a) and $\sum_{k=1}^{\infty}\frac{\iota(k)}{c(k)}<\infty$ with  Remark \ref{Rem4}(c), we obtain $\lim_{k\rightarrow\infty}\iota(k)=0$ and $\lim_{k\rightarrow\infty}\frac{\iota(k)}{c(k)}=0$. According to Lemma 3.1 in \cite{ram2010distributed},
$$\lim_{k\rightarrow\infty}\sum_{s=1}^{k}\iota(s-1)\beta^{k-s}=0,\, \lim_{k\rightarrow\infty}\sum_{s=1}^{k}\frac{\iota(s-1)}{c(s-1)}\beta^{k-s}=0.$$
Thus, the conclusion follows.
%\begin{align}\label{B-9}
%\lim_{k\rightarrow\infty}\mathbb{E}\big\|x_{i}(k+1)-\bar{x}(k+1)\big\|=0,\;\forall i\in \mathcal{N}.
%\end{align}

\section{Proof of Lemma \ref{Lem10}}

\noindent
(a). Define
\begin{align*}
[C_{i}(k)]_{1}=\xi_{i}(k)+c(k)\bigtriangleup_{i}(k)\notag,\\
[C_{i}(k)]_{2}=\xi_{i}(k)-c(k)\bigtriangleup_{i}(k)\notag.\\
\end{align*}
According to Lemma \ref{Lem1},
\begin{align}\label{C-2}
&f_{i}\big([C_{i}(k)]_{1},\lambda_{i}(k)\big)-f_{i}\big([C_{i}(k)]_{2},\lambda_{i}(k)\big)\notag\\\in &\big\langle \partial f_{i_{\xi_{i}(k)+\theta_{i}c(k)\bigtriangleup_{i}(k)}}\big(\xi_{i}(k)+\theta_{i}c(k)\bigtriangleup_{i}(k),\lambda_{i}(k)\big),\notag\\  &2c(k)\bigtriangleup_{i}(k)\big\rangle,
\end{align}
where $\theta_{i}\in[-1,1]$ is a constant. Therefore, there exists $\varsigma_{i} \in   \partial f_{i_{\xi_{i}(k)+\theta_{i}c(k)\bigtriangleup_{i}(k)}}\big(\xi_{i}(k)+\theta_{i}c(k)\bigtriangleup_{i}(k),\lambda_{i}(k)\big)$ such that
$$
f_{i}\big([C_{i}(k)]_{1},\lambda_{i}(k)\big)-f_{i}\big([C_{i}(k)]_{2},\lambda_{i}(k)\big)=\big\langle \varsigma_{i}, \; 2c(k)\bigtriangleup_{i}(k)\big\rangle.
$$
By taking the conditional expectation of $\big\langle d_{i}(k),\xi_{i}(k)-x^{*}\big\rangle$ with respect to $F(k)$, we obtain
\begin{align}\label{C-4}
\mathbb{E}\big[\big\langle d_{i}(k), \xi_{i}(k)-x^*\big\rangle \big | F(k)\big]= D_{i}(k),
\end{align}
with
$
D_{i}(k)=\mathbb{E}\Big[ (\varsigma_{i})^{\top}\bigtriangleup_{i}(k)\big[\bigtriangleup_{i}(k)\big]^{-\top}(\xi_{i}(k)-x^{*})\big | F(k)\Big]
$, which can be further formulated as:
\begin{align}\label{C-5}
D_{i}(k)&=\mathbb{E}\Big[ (\varsigma_{i})^{\top}\Big(\bigtriangleup_{i}(k)\big[\bigtriangleup_{i}(k)\big]^{-\top}-I\Big)(\xi_{i}(k)\notag\\&-x^{*})\big | F(k)\Big]+\mathbb{E}\Big[\big\langle \varsigma_{i}, \;\xi_{i}(k)-x^{*}\big\rangle\big | F(k)\Big].
\end{align}
By Definition \ref{Def1} and Lemma \ref{Lem13}, we obtain
\begin{align}\label{C-6}
&\mathbb{E}[\big\langle \varsigma_{i},\;\xi_{i}(k)-x^{*}\big\rangle|F(k)]
\notag\\=&\mathbb{E}[\big\langle \varsigma_{i},\;\xi_{i}(k)+\theta_{i}c(k)\bigtriangleup_{i}(k)-\theta_{i}c(k)\bigtriangleup_{i}(k)-x^*\big\rangle|F(k)]
\notag\\
\geqslant & \mathbb{E}[f_{i}\big(\xi_{i}(k)+\theta_{i}c(k)\bigtriangleup(k)_{i},\lambda_{i}(k)\big)-f_{i}\big(x^*,\lambda_{i}(k)\big)|F(k)]
\notag\\-&
\big |c(k) \big | L\mathbb{E}\big\|\theta_{i}\bigtriangleup_{i}(k)\big\|
 \notag\\\geqslant &
 \mathbb{E}[f_{i}\big(\xi_{i}(k)+\theta_{i}c(k)\bigtriangleup_{i}(k),\lambda_{i}(k)\big)-f_{i}\big(\bar{x}(k),\lambda_{i}(k)\big)|F(k)]
 \notag\\+&
f_{i}\big(\bar{x}(k),\lambda_{i}(k)\big)-f_{i}\big(x^*,\lambda_{i}(k)\big)-\big |c(k)\big | L\mathbb{E}\big\|\theta_{i} \bigtriangleup_{i}(k)\big\|
\notag\\\geqslant & f_{i}\big(\bar{x}(k),\bar{\lambda}(k)\big)-f_{i}\big(x^*,\lambda^{*}\big)+
f_{i}\big(\bar{x}(k),\lambda_{i}(k)\big)
\notag\\-& f_{i}\big(\bar{x}(k),\bar{\lambda}(k)\big)-f_{i}\big(x^*,\lambda_{i}(k)\big)+f_{i}\big(x^*,\lambda^{*}\big)\notag\\-&L\big\|\xi_{i}(k)-\bar{x}(k)\big\|-2\big|c(k)\big |L\mathbb{E}\big\|\theta_{i}\bigtriangleup_{i}(k)\big\|
\notag\\\geqslant&
 f_{i}\big(\bar{x}(k),\bar{\lambda}(k)\big)-f_{i}\big(x^*,\lambda^{*}\big)-L\big\|\xi_{i}(k)-\bar{x}(k)\big\|\notag\\-&K\big\|\lambda_{i}(k)-\bar{\lambda}(k)\big\|-K\|\lambda_{i}(k)-\lambda^{*}\big\|-2c(k)L\mathbb{E}\big\|\bigtriangleup_{i}(k)\big\|,
\end{align}
and
\begin{align}
&\bigg|\mathbb{E}\Big[ (\varsigma_{i})^{\top}\Big(\bigtriangleup_{i}(k)\big[\bigtriangleup_{i}(k)\big]^{-\top}-I\Big)(\xi_{i}(k)-x^{*})\big | F(k)\Big]\bigg |\notag\\=&\bigg |\mathbb{E}\Big[ (\varsigma_{i}-\!\!\partial f_{i_{x_{i}(k)}}\big(x_{i}(k),\lambda_{i}(k)\big))^{\top}\Big(\bigtriangleup_{i}(k)\big[\bigtriangleup_{i}(k)\big]^{-\top}\notag\\-&I\Big)(\xi_{i}(k)- x^{*})\big | F(k)\Big]\bigg |\leqslant Bc(k).\label{C-7}
\end{align}
for a positive constant $B$.
Combining \eqref{C-6}, \eqref{B-7} with  \eqref{C-4} gives
\begin{align}\label{C-8}
&\mathbb{E}\big[\big\langle d_{i}(k),\;x_{i}(k)-\xi^*\big\rangle \big|F(k)\big] \notag\\ \geqslant  &
f_{i}\big(\bar{x}(k),\bar{\lambda}(k)\big)-f_{i}\big(x^*,\lambda^{*}\big)-L\big\|\xi_{i}(k)-\bar{x}(k)\big\|-Bc(k)\notag\\-&K\big\|\lambda_{i}(k)-\bar{\lambda}(k)\big\|-K\|\lambda_{i}(k)-\lambda^{*}\big\|-c(k)L\big\|\bigtriangleup_{i}(k)\big\|.
\end{align}

\noindent (b). Similar to the proof of part (a), we get
\begin{align}\label{C-10}
&\mathbb{E}\big[\big\langle d_{i}(k),\;x_{i}(k)-\xi^*\big\rangle \big|F(k)\big] \notag\\ \geqslant  &
f_{i}\big(\bar{x}(k),\lambda^{*}\big)-f_{i}\big(x^*,\lambda^{*}\big)-L\big\|\xi_{i}(k)-\bar{x}(k)\big\|-Bc(k)\notag\\-&2K\|\lambda_{i}(k)-\lambda^{*}\big\|-c(k)L\big\|\bigtriangleup_{i}(k)\big\|.
\end{align}
The proof of the second part of Lemma \ref{Lem10} is completed by taking  the expectation to both sides of \eqref{C-10}.

\ifCLASSOPTIONcaptionsoff
  \newpage
\fi

\begin{thebibliography}{10}
\providecommand{\url}[1]{#1}
\csname url@samestyle\endcsname
\providecommand{\newblock}{\relax}
\providecommand{\bibinfo}[2]{#2}
\providecommand{\BIBentrySTDinterwordspacing}{\spaceskip=0pt\relax}
\providecommand{\BIBentryALTinterwordstretchfactor}{4}
\providecommand{\BIBentryALTinterwordspacing}{\spaceskip=\fontdimen2\font plus
\BIBentryALTinterwordstretchfactor\fontdimen3\font minus
  \fontdimen4\font\relax}
\providecommand{\BIBforeignlanguage}[2]{{%
\expandafter\ifx\csname l@#1\endcsname\relax
\typeout{** WARNING: IEEEtran.bst: No hyphenation pattern has been}%
\typeout{** loaded for the language `#1'. Using the pattern for}%
\typeout{** the default language instead.}%
\else
\language=\csname l@#1\endcsname
\fi
#2}}
\providecommand{\BIBdecl}{\relax}
\BIBdecl

\bibitem{nedic2009distributed}
A.~Nedic and A.~Ozdaglar, ``Distributed subgradient methods for multi-agent
  optimization,'' \emph{IEEE Transactions on Automatic Control}, vol.~54,
  no.~1, pp. 48--61, 2009.

\bibitem{yi2015distributed}
P.~Yi, Y.~Hong, and F.~Liu, ``Distributed gradient algorithm for constrained
  optimization with application to load sharing in power systems,''
  \emph{Systems \& Control Letters}, vol.~83, pp. 45--52, 2015.

\bibitem{nedic2010constrained}
A.~Nedic, A.~Ozdaglar, and P.~A. Parrilo, ``Constrained consensus and
  optimization in multi-agent networks,'' \emph{IEEE Transactions on Automatic
  Control}, vol.~55, no.~4, pp. 922--938, 2010.

\bibitem{ram2010distributed}
S.~S. Ram, A.~Nedi{\'c}, and V.~V. Veeravalli, ``Distributed stochastic
  subgradient projection algorithms for convex optimization,'' \emph{Journal of
  Optimization Theory and Applications}, vol. 147, no.~3, pp. 516--545, 2010.

\bibitem{zhang2015distributed}
Y.~Zhang, Y.~Lou, Y.~Hong, and L.~Xie, ``Distributed projection-based
  algorithms for source localization in wireless sensor networks,'' \emph{IEEE
  Transactions on Wireless Communications}, vol.~14, no.~6, pp. 3131--3142,
  2015.

\bibitem{cherukuri2015distributed}
A.~Cherukuri and J.~Cort{\'e}s, ``Distributed generator coordination for
  initialization and anytime optimization in economic dispatch,'' \emph{IEEE
  Transactions on Control of Network Systems}, vol.~2, no.~3, pp. 226--237,
  2015.

\bibitem{yuan2016regularized}
D.~Yuan, D.~W. Ho, and S.~Xu, ``Regularized primal-dual subgradient method for
  distributed constrained optimization.'' \emph{IEEE Trans. Cybernetics},
  vol.~46, no.~9, pp. 2109--2118, 2016.

\bibitem{yuan2017adaptive}
D.~Yuan, D.~W. Ho, and G.-P. Jiang, ``An adaptive primal-dual subgradient
  algorithm for online distributed constrained optimization,'' \emph{IEEE
  Transactions on Cybernetics}, 2017.

\bibitem{chen1999kiefer}
H.-F. Chen, T.~E. Duncan, and B.~Pasik-Duncan, ``A {K}iefer-{W}olfowitz
  algorithm with randomized differences,'' \emph{IEEE Transactions on Automatic
  Control}, vol.~44, no.~3, pp. 442--453, 1999.

\bibitem{conn2009introduction}
A.~R. Conn, K.~Scheinberg, and L.~N. Vicente, \emph{Introduction to
  derivative-free optimization}.\hskip 1em plus 0.5em minus 0.4em\relax SIAM,
  2009, vol.~8.

\bibitem{duchi2013optimal}
J.~C. Duchi, M.~I. Jordan, M.~J. Wainwright, and A.~Wibisono, ``Optimal rates
  for zero-order convex optimization: The power of two function evaluations,''
  \emph{IEEE Transactions on Information Theory}, vol.~61, no.~5, pp.
  2788--2806, 2013.

\bibitem{nesterov2011random}
Y.~Nesterov and V.~Spokoiny, ``Random gradient-free minimization of convex
  functions,'' Universit{\'e} catholique de Louvain, Center for Operations
  Research and Econometrics (CORE), Tech. Rep., 2011.

\bibitem{yuan2015randomized}
D.~Yuan and D.~W. Ho, ``Randomized gradient-free method for multiagent
  optimization over time-varying networks,'' \emph{IEEE Transactions on neural
  networks and learning systems}, vol.~26, no.~6, pp. 1342--1347, 2015.

\bibitem{wang2018distributed}
Y.~Wang, P.~Lin, and Y.~Hong, ``Distributed regression estimation with
  incomplete data in multi-agent networks,'' \emph{Science China Information
  Sciences}, vol.~61, no.~9, p. 092202, 2018.

\bibitem{wang2017distributed}
Y.~Wang, P.~Lin, and H.~Qin, ``Distributed classification learning based on
  nonlinear vector support machines for switching networks,''
  \emph{Kybernetika}, vol.~53, no.~4, pp. 595--611, 2017.

\bibitem{li2017distributed}
H.~Li, C.~Huang, G.~Chen, X.~Liao, and T.~Huang, ``Distributed consensus
  optimization in multiagent networks with time-varying directed topologies and
  quantized communication,'' \emph{IEEE Transactions on Cybernetics}, vol.~47,
  no.~8, pp. 2044--2056, 2017.

\bibitem{wu2012comparison}
L.~Wu, M.~Shahidehpour, and Z.~Li, ``Comparison of scenario-based and interval
  optimization approaches to stochastic scuc,'' \emph{IEEE Transactions on
  Power Systems}, vol.~27, no.~2, pp. 913--921, 2012.

\bibitem{neumaier1990interval}
A.~Neumaier, \emph{Interval methods for systems of equations}.\hskip 1em plus
  0.5em minus 0.4em\relax Cambridge university press, 1990, vol.~37.

\bibitem{rohn1994positive}
J.~Rohn, ``Positive definiteness and stability of interval matrices,''
  \emph{SIAM Journal on Matrix Analysis and Applications}, vol.~15, no.~1, pp.
  175--184, 1994.

\bibitem{levin1999nonlinear}
V.~Levin, ``Nonlinear optimization under interval uncertainty,''
  \emph{Cybernetics and Systems Analysis}, vol.~35, no.~2, pp. 297--306, 1999.

\bibitem{hu2006novela}
B.~Q. Hu and S.~Wang, ``A novel approach in uncertain programming part i: New
  arithmetic and order relation for interval numbers,'' \emph{Journal of
  Industrial \& Management Optimization}, vol.~2, no.~4, pp. 351--371, 2006.

\bibitem{hisao1990multiobjective}
I.~Hisao and T.~Hideo, ``Multiobjective programming in optimization of the
  interval objective function,'' \emph{European Journal of Operational
  Research}, vol.~48, no.~2, pp. 219--225, 1990.

\bibitem{wu2008interval}
H.-C. Wu, ``On interval-valued nonlinear programming problems,'' \emph{Journal
  of Mathematical Analysis and Applications}, vol. 338, no.~1, pp. 299--316,
  2008.

\bibitem{liu2007numerical}
S.-T. Liu and R.-T. Wang, ``A numerical solution method to interval quadratic
  programming,'' \emph{Applied mathematics and computation}, vol. 189, no.~2,
  pp. 1274--1281, 2007.

\bibitem{jiang2008nonlinear}
C.~Jiang, X.~Han, G.~Liu, and G.~Liu, ``A nonlinear interval number programming
  method for uncertain optimization problems,'' \emph{European Journal of
  Operational Research}, vol. 188, no.~1, pp. 1--13, 2008.

\bibitem{jayswal2011sufficiency}
A.~Jayswal, I.~Stancu-Minasian, and I.~Ahmad, ``On sufficiency and duality for
  a class of interval-valued programming problems,'' \emph{Applied Mathematics
  and Computation}, vol. 218, no.~8, pp. 4119--4127, 2011.

\bibitem{hladik2012interval}
M.~Hlad{\i}k, ``Interval linear programming: A survey,'' \emph{Linear
  Programming-New Frontiers in Theory and Applications}, pp. 85--120, 2012.

\bibitem{bhurjee2012efficient}
A.~K. Bhurjee and G.~Panda, ``Efficient solution of interval optimization
  problem,'' \emph{Mathematical Methods of Operations Research}, vol.~76,
  no.~3, pp. 273--288, 2012.

\bibitem{bellet2015distributed}
A.~Bellet, Y.~Liang, A.~B. Garakani, M.-F. Balcan, and F.~Sha, ``A distributed
  {F}rank-{W}olfe algorithm for communication-efficient sparse learning,'' in
  \emph{Proceedings of the 2015 SIAM International Conference on Data
  Mining}.\hskip 1em plus 0.5em minus 0.4em\relax SIAM, 2015, pp. 478--486.

\bibitem{anit2018distributed}
K.~S. Anit, J.~Dusan, B.~Dragana, and K.~Soummya, ``Distributed zeroth order
  optimization over random networks: A {K}iefer-{W}olfowitz stochastic
  approximation approache,'' \emph{arXiv:1803.07836}, 2018.

\bibitem{hajinezhad2017zeroth}
D.~Hajinezhad, M.~Hong, and A.~Garcia, ``Zeroth order nonconvex multi-agent
  optimization over networks,'' \emph{arXiv:1710.09997}, 2017.

\bibitem{yuan2015zeroth}
D.~Yuan, D.~W. Ho, and S.~Xu, ``Zeroth-order method for distributed
  optimization with approximate projections.'' \emph{IEEE Transactions on
  Neural Networks and Learning Systems}, vol.~27, no.~2, pp. 284--294, 2015.

\bibitem{hiriart2012fundamentals}
J.-B. Hiriart-Urruty and C.~Lemar{\'e}chal, \emph{Fundamentals of Convex
  Analysis}.\hskip 1em plus 0.5em minus 0.4em\relax Springer Science \&
  Business Media, 2012.

\bibitem{clarke1998nonsmooth}
F.~H. Clarke, Y.~S. Ledyaev, and R.~J. Stern, \emph{Nonsmooth Analysis and
  Control Theory}.\hskip 1em plus 0.5em minus 0.4em\relax Springer Science \&
  Business Media, 2008.

\bibitem{durrett2010probability}
R.~Durrett, \emph{Probability: Theory and Examples}.\hskip 1em plus 0.5em minus
  0.4em\relax Cambridge university press, 2010.

\bibitem{polyak1987introduction}
B.~T. Polyak, \emph{Introduction to Optimization}.\hskip 1em plus 0.5em minus
  0.4em\relax Chapman and Hall, 1987.

\bibitem{aubin2012differential}
J.-P. Aubin and A.~Cellina, \emph{Differential inclusions: set-valued maps and
  viability theory}.\hskip 1em plus 0.5em minus 0.4em\relax Springer Science \&
  Business Media, 2012, vol. 264.

\bibitem{maeda2012optimization}
T.~Maeda, ``On optimization problems with set-valued objective maps: existence
  and optimality,'' \emph{Journal of Optimization Theory and Applications},
  vol. 153, no.~2, pp. 263--279, 2012.

\bibitem{chen2006stochastic}
H.-F. Chen, \emph{Stochastic Approximation and its Applications}.\hskip 1em
  plus 0.5em minus 0.4em\relax Springer Science \& Business Media, 2006,
  vol.~64.

\bibitem{hazan2016introduction}
E.~Hazan \emph{et~al.}, ``Introduction to online convex optimization,''
  \emph{Foundations and Trends{\textregistered} in Optimization}, vol.~2, no.
  3-4, pp. 157--325, 2016.

\bibitem{nedic2016stochastic}
A.~Nedi{\'c} and A.~Olshevsky, ``Stochastic gradient-push for strongly convex
  functions on time-varying directed graphs,'' \emph{IEEE Transactions on
  Automatic Control}, vol.~61, no.~12, pp. 3936--3947, 2016.

\end{thebibliography}
\end{document}